\documentclass[11pt]{amsart}

\usepackage[a4paper,margin=1.05in]{geometry}
\usepackage{amsmath,amssymb,amsthm,mathtools}
\usepackage[expansion=false]{microtype}
\usepackage[hidelinks]{hyperref}
\usepackage[nameinlink,capitalize]{cleveref}
\usepackage{color}

\numberwithin{equation}{section}
\allowdisplaybreaks

\theoremstyle{plain}
\newtheorem{theorem}{Theorem}[section]
\newtheorem{proposition}{Proposition}[section]
\newtheorem{lemma}{Lemma}[section]
\newtheorem{corollary}{Corollary}[section]

\theoremstyle{definition}
  \newtheorem{definition}{Definition}[section]

\theoremstyle{remark}
\newtheorem{remark}{Remark}[section]

\newcommand{\diam}{\operatorname{diam}}
\newcommand{\dist}{\operatorname{dist}}

\newcommand{\tr}{\operatorname{tr}}

\newcommand{\beq}{\begin{equation}}
\newcommand{\eeq}{\end{equation}}

\begin{document}

\title[Three-point estimates with Dini data]
{A new approach to interior estimates for fully nonlinear elliptic equations}
\author{Jui-Yun Hung, Jiakun Liu and Neil Trudinger}
 
\address{Jui-Yun Hung,
Centre for Mathematics and its Applications, The Australian National University, Canberra, ACT 0200, Australia}
\email{Jui-Yun.Hung@anu.edu.au}
 
\address{Jiakun Liu,
School of Mathematics and Statistics, The University of Sydney, Camperdown, NSW 2006, Australia}
\email{jiakun.liu@sydney.edu.au} 

\address{Neil Trudinger, 
Centre for Mathematics and its Applications, The Australian National University, Canberra, ACT 0200, Australia}
\email{neil.trudinger@anu.edu.au}

\thanks{This research was supported by the Australian Research Council through grants FT220100368 and DP230100499.}



\begin{abstract}
We give new proofs of interior second derivative estimates for the quadratic Hessian and Monge-Amp\`ere equations with Dini continuous right-hand sides. 
Our approach uses a three-point maximum principle and requires no differentiation of the equations. 
For the quadratic Hessian equation, the estimate holds in every dimension $n\ge2$.
For the Monge-Amp\`ere equation, the proof is localised on sublevel sets of a strictly convex solution. 
\end{abstract}

\date{25 October 2026}

\maketitle

 \baselineskip15pt
 \parskip4pt

\section{Introduction}

The known two-point maximum principle is a powerful tool in deriving gradient estimates for geometric heat equations, which has its origins in the study of parabolic equations in the 1960s.  The approach also became fundamental in the theory of viscosity solutions, following its introduction by Ishii in the 1980s \cite{CIL92} and was also used by Trudinger \cite{Tr88} to prove H\"older gradient regularity of solutions of fully nonlinear uniformly elliptic equations, such as the Isaacs equation from stochastic game theory. 
In the 1990s, the method was further developed for applications to geometric flows by Hamilton \cite{Ha92-1, Ha92-2} and Huisken \cite{Hu98}. More recent successful applications include Andrews and Clutterbuck's proof of the fundamental gap conjecture \cite{AC11}, Brendle's proof of the Lawson conjecture \cite{Br13}, and Andrews and Li's proof of the Pinkall-Sterling conjecture \cite{AL15}. 
For further background and related developments, we refer the reader to the excellent surveys \cite{An15,Br14}. 

In \cite{HLT}, we combined the two-point maximum principle with a localisation penalty to establish interior second derivative estimates for the quadratic Hessian and Monge-Amp\`ere equations with positive, continuously differentiable right-hand sides. 
Our approach involved differentiating the equation once and applying the two-point maximum principle to the linearised equation. This yielded a modulus of continuity for the gradient, from which the desired second derivative estimate follows by passing to the limit. 

In this paper, we consider more general right-hand sides that depend on the solution and its gradient and are merely Dini continuous in the variable $x$. 
By replacing derivatives with difference quotients and using a new three-point maximum principle, we establish interior second derivative estimates without differentiating the equation. 
The corresponding Schauder estimates then follow as a corollary.

Consider the quadratic Hessian equation, generally known as the $\sigma_2$ equation, 
	\begin{equation}\label{sig2eq}
		\sigma_2\big(\lambda(D^2u)\big) = \sum_{1\leq i < j \leq n} \lambda_i\lambda_j = f,
	\end{equation}
where $f$ is a given positive function and $\lambda_1,\cdots,\lambda_n$ are the eigenvalues of $D^2u$. 
When $f=1$, the interior $C^{1,1}$ estimate was established by Warren and Yuan \cite{WY09} in dimension three.
Shankar and Yuan \cite{SY25} obtained the estimate in dimension four and, in higher dimensions, under the additional condition
	\begin{equation*}
		\lambda_{\min}(D^2u) \geq -c(n)\Delta u. 
	\end{equation*}
Just recently, Li and Wu \cite{LWarX} discovered the interior $C^{1,1}$ estimate in arbitrary dimension $n\geq 2$. 
Specifically, they proved that for any $u\in C^\infty(B_2)$ satisfies \eqref{sig2eq} in $B_2\subset\mathbb{R}^n$,
	\begin{equation}\label{C11e}
		\|D^2u\|_{L^\infty(B_{1/2})} \leq C\left(n, \|u\|_{L^\infty(B_2)}\right).
	\end{equation}
Their new and ingenious argument proceeds through a reduction to the Pogorelov estimate of Chou and Wang \cite{CW01}, rather than a Jacobi inequality or a Hessian doubling argument used in the lower dimensional cases. 
They also remarked on the extension to general positive $f\in C^{1,1}$ in \cite[Remark 1.3]{LWarX}.
For related results and extensions to general positive inhomogeneous terms $f$, we refer the reader, without attempting to be exhaustive, to \cite{Fan26, GQ19, Mo25, Qiu24} and references therein.

Using a three-point maximum principle, we establish the following result for the quadratic Hessian equation.
\begin{theorem}\label{thm:main}
Let $n\ge2$, and let $u$ be a $C^2$ function defined in a neighbourhood of $\overline B_1$.
Suppose that $\lambda(D^2u)\in\Gamma_2$ and that $u$ satisfies
	\begin{equation}\label{eq:PDE}
 		\sigma_2\big(\lambda(D^2u)\big)=f(\cdot,u,Du) \quad\text{ in } B_1,
	\end{equation}
where, for some constants $Q,M\ge1$,
	\begin{equation}\label{eq:norm}
		 1\le f(x,u(x),Du(x))\le Q \quad \text{ and }\quad |Du(x)|\le M	 \quad\forall\, x\in B_1.
	\end{equation}
	
Assume that $f$ satisfies
	\begin{equation}\label{eq:hypothesis}
		 |f(x,z,p)-f(y,w,q)| \le \delta_x(|x-y|)+L_u|z-w|+L_p|p-q|
	\end{equation}
for all $x,y\in B_1$, $z,w\in u(\overline B_1)$ and $p,q\in Du(\overline B_1)$, where $L_u,L_p\ge0$ and $\delta_x$ is a continuous, nondecreasing Dini modulus of continuity satisfying $\delta_x(0)=0$ and $\int_0^2\frac{\delta_x(r)}{r}\,dr<\infty$.

Then there exists a constant
$C=C(n,Q,M,L_u,L_p,\int_0^2\frac{\delta_x(r)}{r}\,dr)$ such that
	\begin{equation}\label{eq:theorem_conclusion}
 		|D^2u|_{L^\infty(B_{1/2})}\le C.
	\end{equation}
\end{theorem}

In particular, when $f\equiv1$, applying this estimate on interior balls and using the gradient estimate in \cite{Tr97}, we obtain the same conclusion with $C$ depending only on $n$ and $\|u\|_{L^\infty(B_1)}$. 
Once having \eqref{eq:theorem_conclusion}, equation \eqref{eq:PDE} becomes uniformly elliptic, and we have the corresponding Schauder estimate
\begin{corollary}
Under the hypotheses of Theorem \ref{thm:main}, suppose in addition that $u\in C^3(B_1)$ and that the function $g(x)=f(x,u,Du)$ belongs to $C^\alpha(B_1)$ for some $\alpha\in(0,1)$. Then
	\begin{equation}\label{Sig2Schauder}
 		\|u\|_{C^{2,\alpha}(B_{1/2})}\le C,
	\end{equation}
where $C$ depends in addition on $\alpha$, $\|u\|_{L^\infty(B_1)}$ and $\|g\|_{C^\alpha(B_1)}$.
\end{corollary}
We remark that interior Hessian bounds and $C^{2,\alpha}$ estimates for the quadratic Hessian equation with a positive right-hand side $f=f(x)\in C^\alpha$ were also recently obtained by Chen, Zhou and Zhu \cite{CZZ}. Their argument, however, relies on H\"older continuity and does not directly yield the Hessian estimate under the Dini continuity assumption considered here, see Remark 4.2 of \cite{CZZ}.

In dimension two, \eqref{eq:PDE} is the Monge-Amp\`ere equation. 
We now consider the Monge-Amp\`ere equation in arbitrary dimension 
	\begin{equation}\label{MAeq}
		\det D^2u  = f \quad\text{ in } \Omega\subset\mathbb{R}^n,
	\end{equation}
where $\Omega$ is a bounded domain and $f$ is a given positive function. 
Pogorelov \cite{Po78} established an interior $C^2$ estimate for convex solutions with affine boundary data, now known as Pogorelov's estimate, see \cite{GT01,LT10}. 
Specifically, for any $\Omega'\Subset\Omega$,
	\begin{equation}\label{Pogo}
		\sup_{\Omega'} |D^2u| \leq \frac{C}{d_{\Omega'}},
	\end{equation}
where $d_{\Omega'}=\dist(\Omega',\partial\Omega)$ and $C$ depends on $n, \|\log f\|_{C^{1,1}(\Omega)}, \|u\|_{C^1(\Omega)}$, and $\diam\Omega$. 
In dimension two, a genuine interior $C^2$ estimate was first established by Heinz \cite{He59} and recently obtained by different methods in \cite{CHO16,Liu21}.
For the constant right-hand side $f=1$, Yuan \cite{Yu23} gave a new proof of Pogorelov's estimate using a mean value inequality approach, rooted in Trudinger's classic proof of the gradient estimates for the minimal hypersurface equation \cite{GT01}.

By localising the three-point maximum principle on sublevel sets of a strictly convex solution, 
We establish the following estimate for the Monge-Amp\`ere equation.
\begin{theorem}\label{thm:ma}
Let $n\ge2$, and let $\Omega\subset\mathbb R^n$ be a bounded convex domain satisfying $B_1\subset\Omega\subset B_n$. 
Suppose that $u\in C^2(\Omega)\cap C^0(\overline\Omega)$ is a strictly convex solution of
	\begin{equation}\label{eq:ma_PDE}
 	\begin{split}
		\det D^2u &= f(\cdot,u,Du)\quad\text{in }\Omega, \\
 		u &=0 \quad\text{on }\partial\Omega,
	\end{split}
	\end{equation}
where $0<\lambda\le f(\cdot,u,Du)\le\Lambda$ in $\Omega$.

Let $\Omega'\Subset\Omega''\Subset\Omega$. Assume that $G(x,z,p)=\log f(x,z,p)$ satisfies
	\begin{equation}\label{eq:ma_hypothesis}
 		|G(x,z,p)-G(y,w,q)| \le\delta_x(|x-y|)+L_u|z-w|+L_p|p-q|
	\end{equation}
for all $x,y\in\Omega''$, $z,w\in u(\Omega'')$ and $p,q\in Du(\Omega'')$, where $L_u,L_p\ge0$ and $\delta_x$ is a continuous, nondecreasing Dini modulus of continuity satisfying $\delta_x(0)=0$ and $\int_0^{\operatorname{diam}\Omega''} \frac{\delta_x(r)}r\,dr<\infty$.

Then there exists a constant
	\[ C=C\left(n,\lambda,\Lambda,\Omega',\Omega'', \|Du\|_{L^\infty(\Omega'')},L_u,L_p, \int_0^{\operatorname{diam}\Omega''} \frac{\delta_x(r)}r\,dr\right) \]
such that
	\begin{equation}\label{eq:ma_conclusion}
		 |D^2u|_{L^\infty(\Omega')}\le C.
	\end{equation}
\end{theorem}

By the gradient estimate and the convexity of $u$, the dependence of $C$ on $\|u\|_{C^1(\Omega'')}$ can be replaced by dependence on $\dist(\Omega'',\partial\Omega)$, $f$, and $\Omega$. 
Similarly, we have the following Schauder estimate as a corollary.
\begin{corollary}
Under the hypotheses of Theorem \ref{thm:ma}, suppose in addition that $u\in C^3(\Omega)$ and that the function $g(x)=f(x,u,Du)$ belongs to $C^\alpha(\Omega)$ for some $\alpha\in(0,1)$. Then
	\begin{equation}\label{maSchauder}
 		\|u\|_{C^{2,\alpha}(\Omega')}\le C,
	\end{equation}
where $C$ depends in addition on $\alpha$ and $\|g\|_{C^\alpha(\Omega)}$.
\end{corollary}
We remark that the interior $C^{2,\alpha}$ estimate for the Monge-Amp\`ere equation was established by Caffarelli in his celebrated paper \cite{Ca90}. Jian and Wang \cite{JW07} later gave a simplified proof based on a perturbation argument. 

\begin{remark}
We also remark that through a modification of the proof of boundary $C^2$ estimates in \cite{ JT1},  global $C^2$ estimates for oblique boundary value problems for quadratic Hessian  equations of the general form \eqref{eq:PDE} in bounded mean convex domains, (that is bounded domains with non-negative boundary mean curvature), follow from the corresponding interior Hessian estimates in Theorem 1.1 and the details, as well as more general results, will also be provided in an ensuing paper by the authors of \cite{ JT1}.  Here the function $f$,  as in \cite{JT1}, is sufficiently smooth with $\sqrt{f}$ convex in the gradient variables.
\end{remark}

The proofs of Theorems \ref{thm:main} and \ref{thm:ma} follow a unified approach based on a three-point maximum principle, in which the middle point is allowed to move along the segment joining two endpoints. 
A key ingredient is a three-point test function consisting of a weighted difference quotient of $u$, a concave distance profile and a localisation term. The construction of this test function is tailored to each equation.
The paper is organised as follows. We prove Theorems \ref{thm:main} and \ref{thm:ma} in Sections 2 and 3, respectively.
The auxiliary matrix inequalities and computations used in the proofs are collected in the appendices.

\noindent{\bf Acknowledgment:} 
We thank the Mathematical Research Institute MATRIX in Australia, where part of this research was performed. 
We are particularly grateful to Professor Ben Andrews for introducing the second author to the two-point maximum principle at the MATRIX-RIMS Tandem Workshop in March 2024 and for many inspiring discussions.

\noindent{\bf AI disclosure:}
The research questions and the strategy of applying two-point and multi-point comparison principles to obtain the interior estimates in this paper originated with the authors. 
In particular, the authors proposed incorporating localisation terms to exclude positive boundary maxima and using suitable second-order finite differences to construct the comparison functions and obtain the second derivative estimate.
Under their direction, Astra (OpenAI) assisted with mathematical computations for the quadratic Hessian and Monge-Amp\`ere equations. 
This assistance included constructing explicit profiles $\omega$ satisfying conditions specified by the authors, deriving and checking first- and second-variation formulas, and working through the matrix inequalities used in this paper. 
In developing the comparison functions, the authors also asked Astra to search for local obstructions: 
local solutions near the comparison points for which a candidate comparison function has a positive local maximum with negative definite Hessian. 
These tests helped identify limitations of proposed multi-point functions and informed the authors’ refinement of the comparison functions. 
OpenAI tools also assisted with language editing, LaTeX formatting, and organisation of the presentation. 
The authors have reviewed the complete manuscript and take sole responsibility for all mathematical statements and proofs, the claims of originality, and the relation of the results to the existing literature.

\section{Interior estimate for the quadratic Hessian equation}

Let $n\ge2$ and $\lambda=(\lambda_1,\lambda_2,\cdots,\lambda_n)\in\mathbb{R}^n$. Recall the definitions (see \cite{TW99,Wa09})
	\begin{align*}
		\sigma_1(\lambda) &:= \sum_{i=1}^n\lambda_i, \qquad \sigma_2(\lambda):= \sum_{1\leq i<j\leq n}\lambda_i\lambda_j, \\
		\Gamma_2 &:= \big\{ \lambda\in\mathbb{R}^n \,:\, \sigma_1(\lambda)>0,\, \sigma_2(\lambda)>0\big\}.
	\end{align*}
Suppose $u\in C^2$ in a neighbourhood of $\overline{B_1}$, with $\lambda(D^2u)\in\Gamma_2$ and
	\[ \sigma_2\big(\lambda(D^2u)\big)=f(\cdot,u,Du) \quad\text{ in } B_1. \]
In this section, we derive a bound for $|D^2u(0)|$ depending only on $n$, $\|Du\|_{L^\infty(B_1)}$, and the Dini norm of $f$. 
Our argument uses a three-point maximum principle and does not require differentiation of the equation. 

The proof is organised into three parts. In \S\ref{SS21}, we construct a three-point test function $Z$ and show that $Z\leq0$ implies the desired second derivative estimate. To establish this inequality, we argue by contradiction, assuming that $Z$ attains a positive interior maximum. In \S\ref{SS22}, we derive estimates by considering different variations at this maximum. In \S\ref{SS23}, we combine these estimates to obtain a contradiction and complete the proof.

\subsection{Three-point function}\label{SS21}
Let $z,x,y$ be three distinct collinear points with $x$ strictly between $z$ and $y$. Thus, there exists a unit vector $\nu\in\mathbb{S}^{n-1}$ such that
	\begin{equation}\label{eq:chord}
 		y=x+s\nu,\qquad z=x-t\nu \qquad\text{ for } s,t>0.
	\end{equation}
Denote $d=s+t=|y-z|$.
On the space of collinear triples with $y,z\in\overline B_1$ and
$d\le\frac14$, define
	\begin{equation}\label{eq:Z}
 		Z(x,\nu,s,t)=\frac td u(y)+\frac sd u(z)-u(x) -\frac{st}{2d}\bigl[\omega(d)+\kappa(|y|^2+|z|^2)\bigr].
	\end{equation}
This test function combines three ingredients: a weighted finite difference of $u$, whose normalised limit yields a second directional derivative; a distance profile function $\omega$, designed to absorb the Dini oscillation of $f$; and a localisation term $\kappa(|y|^2+|z|^2)$, which excludes that positive maxima of $Z$ occurs on the boundary. 
The profile $\omega$ and the constant $\kappa$ will be chosen below. 

For brevity, write $g(x)=f(x,u(x),Du(x))$. From \eqref{eq:norm}, $1\leq g\leq Q$ and $|Du|\leq M$ in $B_1$. 
Choose the constants 
	\begin{equation}\label{eq:constants}
 		\kappa=4M+1,\qquad C_0=12Q,\qquad
		 \beta=60C_0^2\bigl(\kappa n+1+\sqrt Q\,L_p\bigr).
	\end{equation}
In addition, we define
	\begin{equation}\label{eq:delta}
 		\delta(r)=\delta_x(r)+L_uMr+3L_p\kappa r^2, \qquad 0\le r\le\frac12.
	\end{equation}
This is also a nondecreasing Dini modulus satisfying $\int_0^{1/2}\frac{\delta(r)}{r}\,dr<\infty$. 

Define the three functions, for $0<r\le\frac14$,
\begin{equation}\label{eq:b}
 \begin{aligned}
 b(r)&=6C_0^2\sqrt Q\,\frac{\delta(2r)}r,\\
 b_1(r)&=\int_0^r b(\tau)\,d\tau,\\
 b_2(r)&=\int_0^r e^{-b_1(\tau)}\,d\tau.
 \end{aligned} 
\end{equation}
The Dini condition ensures that these integrals are well defined. Both $b_1$ and $b_2$ extend continuously to $r=0$, with $b_1(0)=b_2(0)=0$. We now define the profile
	\begin{equation}\label{eq:omega}
 		\omega(r)=\frac1\beta\log\bigl(1+\beta A b_2(r)\bigr), \qquad 0\le r\le\frac14,
	\end{equation}
where the constant $A$ is chosen to be
	\begin{equation}\label{eq:A}
 		A=e^{4\beta M+b_1(1/4)} \left(4\kappa+\sqrt Q+1+\frac8\beta\right).
	\end{equation}

The choices of constants and functions in \eqref{eq:constants}--\eqref{eq:A} will be justified in the subsequent calculations. 
We first record some properties of $\omega$ that will be used later. 
Since $b_1$ is nondecreasing, differentiation gives $\omega'(0)=A$, and
	\begin{equation}\label{eq:riccati}
	 \begin{aligned}
		 \omega'(r)&=Ae^{-b_1(r)-\beta\omega(r)}>0,\\
		 \omega''(r)&=-\beta\omega'(r)^2-b(r)\omega'(r)<0 \qquad(r>0). 
	 \end{aligned}
	\end{equation}
Again since $b_1$ is nondecreasing, $b_2(r)\ge r e^{-b_1(r)}$, we have
	\begin{equation}\label{eq:riccati-m}
	\begin{split}
 		\beta r\omega'(r) &= \frac{\beta Ar e^{-b_1(r)}}{1+\beta A b_2(r)} \\
			&\leq \frac{\beta A b_2(r)}{1+\beta A b_2(r)}  <1. 
	\end{split}
	\end{equation}

Also, $b_2(1/4)\ge \frac14e^{-b_1(1/4)}$ and the choice of $A$ in \eqref{eq:A} imply
	\begin{align*}
		\beta A b_2(1/4) &\geq \frac{\beta}{4} e^{4\beta M+b_1(1/4)} \left(4\kappa+\sqrt Q+1+\frac8\beta\right) e^{-b_1(1/4)} \\
			&= e^{4\beta M} \left[\frac{\beta}{4}\left(4\kappa+\sqrt Q+1\right) +2  \right] > 2e^{4\beta M}.
	\end{align*}
Hence, substituting it into \eqref{eq:omega} we obtain
	\begin{equation}\label{eq:profile_height}
	\begin{split}
		 \omega(\frac14) &=\frac1\beta\log\bigl(1+\beta A b_2(1/4)\bigr) \\
		 	&> \frac1\beta\log\bigl(1+2e^{4\beta M}\bigr) > \frac1\beta\log\bigl(e^{4\beta M}\bigr) = 4M.
	\end{split}
	\end{equation}
Whenever $\omega(r)<4M$, from \eqref{eq:riccati} and \eqref{eq:A} we have
	\begin{equation}\label{eq:profile_slope}
		 \omega'(r)>Ae^{-b_1(1/4)-4\beta M} =4\kappa+\sqrt Q+1+\frac8\beta >\max\{4\kappa,\sqrt Q,1\}.
	\end{equation}
Only these properties of the profile $\omega$ will be used below.

We first explain how the desired estimate follows once $Z\le0$ has been established. 
\begin{proposition}\label{pp2.1}
Let $x=0$, $y=r\nu$, $z=-r\nu$. If $Z\le0$ for any $\nu\in\mathbb{S}^{n-1}$, $r\in(0,1/8)$, then
	\begin{equation}\label{eq:center}
		 |D^2u(0)|\le(n-1)A.
	\end{equation}
\end{proposition}

\begin{proof}
Substituting $s=t=r$ and $d=2r$ into \eqref{eq:Z} and using $Z\leq0$, we get
	\[ \frac{u(r\nu)+u(-r\nu)-2u(0)}2 \le\frac r4\bigl[\omega(2r)+2\kappa r^2\bigr]. \]
Dividing both sides by $r^2/2$ and letting $r\to0$, we obtain
	$$ u_{\nu\nu}(0)\le\omega'(0)=A \quad \forall\,\nu\in\mathbb{S}^{n-1}. $$
Since $\Delta u(0)>0$, the other eigenvalues are bounded below by
$-(n-1)A$. Thus, \eqref{eq:center} follows.
\end{proof}

By rescaling and applying Proposition \ref{pp2.1} on interior balls centred at $x\in B_{1/2}$, we obtain the desired estimate \eqref{eq:theorem_conclusion}. 
It therefore suffices to prove $Z\le0$ in the admissible domain $\mathcal{D}$, consisting of all collinear triples $(x,y,z)$ satisfying \eqref{eq:chord}, with $y,z\in\overline B_1$ and $d=|y-z|\le\frac14$. 

We first show that $Z$ cannot attain a positive maximum on the boundary of this domain.
Since $|Du|\le M$, we have
	\[ \frac td u(y)+\frac sd u(z)-u(x) = \frac td \big(u(y)-u(x)\big) + \frac sd \big(u(z)-u(x)\big) \le\frac{2Mst}{d}. \]
Consequently,
	\begin{equation}\label{eq:boundary}
 		Z(x,y,z)\le\frac{st}{2d} \bigl[4M-\omega(d)-\kappa(|y|^2+|z|^2)\bigr].
	\end{equation}
If either one of $y, z$ lies on $\partial B_1$, then
	\[ Z \leq \frac{st}{2d} \bigl(4M-\kappa\bigr) < 0, \quad\text{since $\kappa>4M$ by \eqref{eq:constants}. }\]
If $d=\frac14$, then
	\[ Z\le\frac{st}{2d} \bigl(4M-\omega(\frac14)\bigr) < 0, \qquad\text{due to \eqref{eq:profile_height}. }\] 
Note that although $Z$ is initially defined for $d=s+t>0$, we can extend it continuously by zero when either $s$ or $t$ tends to zero.
Indeed, since $0\leq\omega(d)\leq\omega(1/4)$, similarly as for \eqref{eq:boundary}, we have
	\[ |Z| \leq \frac{st}{s+t} \Bigl[2M+\frac 12\omega(\frac14)+\kappa\Bigr] \leq C\min\{s,t\}.\]
Thus $Z\to0$ whenever $s\to0$ or $t\to0$, including when both tend to zero. 

Therefore, we conclude that
\begin{proposition}
 If $Z$ takes a positive value, it attains a positive maximum in the interior of the admissible domain $\mathcal{D}$, at a triple of distinct collinear points $(x,y,z)$.
\end{proposition}

All subsequent calculations are carried out at this maximising triple. 
Following \cite{TW99,Wa09}, we denote the $\sigma_2$ operator and its linearised coefficients by
	\begin{equation} \label{eq:sigma2}
	\begin{split}
 		S_2[u] &= \sigma_2(\lambda(D^2u)) = \frac12\bigl[(\Delta u)^2-|D^2u|^2\bigr],\\  
		S_2^{ij}[u] &= \frac{\partial S_2[u]}{\partial u_{ij}} = (\Delta u)\delta_{ij}-u_{ij},  
	\end{split}
	\end{equation} 
where $|D^2u|^2=\sum_{i,j}u_{ij}^2$.  
Note that if $\lambda(D^2u)\in\Gamma_2$, the matrix $(S_2^{ij}[u])$ is positive definite.
We shall use the identities
	\begin{equation}\label{eq:homogeneity}
 		S_2^{ij}[u]u_{ij}=2S_2[u], \qquad \sum_iS_2^{ii}[u]=(n-1)\Delta u.
	\end{equation}
Repeated indices are summed unless otherwise indicated.

\subsection{Estimates by variations}\label{SS22}

Suppose that $Z$ attains its positive maximum at an interior triple $(x,y,z)$.
By a rotation of coordinates we may assume
	\begin{equation}\label{eq:e1}
 		\nu=e_1,\qquad y=x+se_1,\qquad z=x-te_1.
	\end{equation}
For the rest of the local calculation, write
$S_2^{ij}=S_2^{ij}[u](x)$ and $\Delta u = \Delta u(x)$. Thus,
	\[ S_2^{11}=\Delta u-u_{11}(x)>0, \qquad\sum_iS_2^{ii}=(n-1)\Delta u. \]
From \eqref{eq:boundary} and \eqref{eq:profile_slope}, we have
	\begin{equation}\label{eq:positive}
 		\omega(d)+\kappa(|y|^2+|z|^2)<4M,  \qquad \omega'(d)>\max\{4\kappa,\sqrt Q,1\}.
	\end{equation}

In this subsection, we shall derive estimates from different variations of $Z$ at the maximum triple $(x,y,z)$.
\begin{lemma}
The following lower bound holds at the middle point:
	\begin{equation}\label{eq:middle}
 		u_{11}(x)\ge \omega'(d)>1.
	\end{equation}
\end{lemma}

\begin{proof}
To derive \eqref{eq:middle}, we keep $y,z$ fixed and vary $x$ along $e_1$ such that 
	\[ x\mapsto x+h e_1,\qquad s\mapsto s-h,\qquad t\mapsto t+h. \]
Substituting this variation into $Z$, by differentiation we have
	\begin{align*}
		Z(h) &= \frac{t+h}{d} u(y)+\frac{s-h}{d} u(z)-u(x+he_1) -\frac{(s-h)(t+h)}{2d}\bigl[\omega(d)+\kappa(|y|^2+|z|^2)\bigr], \\
		Z'(h) &= \frac 1d u(y) - \frac 1d u(z) - u_1(x+he_1) - \frac{s-t-2h}{2d}\bigl[\omega(d)+\kappa(|y|^2+|z|^2)\bigr], \\
		Z''(h) &= -u_{11}(x+he_1) + \frac1d \bigl[\omega(d)+\kappa(|y|^2+|z|^2)\bigr].
	\end{align*}
Since $Z''(0)\leq0$, we get
	\begin{equation*} 
 		u_{11}(x)\ge \frac{\omega(d)+\kappa(|y|^2+|z|^2)}d \ge\frac{\omega(d)}d\ge\omega'(d)>1,
	\end{equation*}
where the third inequality follows from the concavity of $\omega$ and $\omega(0)=0$.
\end{proof}

\begin{lemma}
The following upper bounds hold at two endpoints:
	\begin{equation}\label{eq:endpointceiling}
		 u_{11}(y),\ u_{11}(z)\le2\omega'(d).
	\end{equation}
\end{lemma}
\begin{proof}
To derive $u_{11}(y)\le2\omega'(d)$, we vary $y$ along $e_1$ with $x, z$ fixed, namely
	\[ y\mapsto y+ \epsilon e_1,\qquad s\mapsto s+\epsilon,\qquad d\mapsto d+\epsilon. \]
Accordingly,
	\begin{align*}
		Z(\epsilon) = \frac{t}{d+\epsilon} u(y+\epsilon e_1)+\frac{s+\epsilon}{d+\epsilon} u(z)-u(x) -\frac{(s+\epsilon)t}{2(d+\epsilon)}\bigl[\omega(d+\epsilon)+\kappa(|y+\epsilon e_1|^2+|z|^2)\bigr].
	\end{align*}	
The first-order condition $0=\left.\frac{d}{d\epsilon}\right|_{\epsilon=0}Z(\epsilon)$ gives
	\begin{equation}\label{eq:first_lengths}
		0= u_1(y)-\frac{u(y)-u(z)}d-\frac{t}{2d}\bigl[\omega(d)+\kappa(|y|^2+|z|^2)\bigr]-\frac s2\bigl[\omega'(d)+2\kappa y_1\bigr].
	\end{equation}
The second-order condition $0\geq \left.\frac{d^2}{d\epsilon^2}\right|_{\epsilon=0}Z(\epsilon)$ gives
	\begin{equation}\label{eq:uppery}
 		u_{11}(y)\le\omega'(d)+\frac s2\omega''(d) +2\kappa y_1+\kappa s.
	\end{equation}	
Since $|y|\le1$, $s\le1/4$, $\omega''<0$, and $\omega'(d)>4\kappa$, we thus obtain $u_{11}(y)\le2\omega'(d).$

Similarly, varying $z$ along $e_1$ with $x,y$ fixed, the corresponding differentiation gives
	\begin{equation}\label{eq:first_lengtht}
		0= \frac{u(y)-u(z)}d-u_1(z)-\frac{s}{2d}\bigl[\omega(d)+\kappa(|y|^2+|z|^2)\bigr]-\frac t2\bigl[\omega'(d)-2\kappa z_1\bigr], 
	\end{equation}
and 
	\begin{equation}\label{eq:upperz}
 		u_{11}(z)\le\omega'(d)+\frac t2\omega''(d) -2\kappa z_1+\kappa t.
	\end{equation}
Since $|z|\le1$, $t\le1/4$, $\omega''<0$, and $\omega'(d)>4\kappa$, we then obtain $u_{11}(z)\le2\omega'(d).$
\end{proof}
	
\begin{lemma}\label{lem23}
We have the following identities for the gradient differences:
	\begin{align*} 
		Du(y)-Du(z) &= \Big[\frac{\omega(d)+\kappa(|y|^2+|z|^2)}2 +\frac{d\omega'(d)}2+\kappa(s^2+t^2)\Big]e_1 +\kappa(s-t)x, \\
		Du(y)-Du(x) &= \frac sd\bigl(Du(y)-Du(z)\bigr) +\frac{\kappa st}{d}(y+z),\nonumber \\
 		Du(z)-Du(x) &= -\frac td\bigl(Du(y)-Du(z)\bigr) +\frac{\kappa st}{d}(y+z). \nonumber
	\end{align*}
\end{lemma}

\begin{proof}
Adding \eqref{eq:first_lengths} and \eqref{eq:first_lengtht}, we first get 
	\begin{equation}\label{eq:first_lengths_sum}
	\begin{aligned}
		u_1(y)-u_1(z) &=\frac12\bigl[\omega(d)+\kappa(|y|^2+|z|^2)\bigr] +\frac d2\omega'(d)+\kappa(sy_1-tz_1)\\
			&=\frac12\bigl[\omega(d)+\kappa(|y|^2+|z|^2)\bigr]+\frac d2\omega'(d)+\kappa\bigl[(s-t)x_1+s^2+t^2\bigr].
	\end{aligned}
	\end{equation}
For each $\alpha\ge2$, let $\nu(h)$ be a smooth curve on $\mathbb S^{n-1}$
with $\nu(0)=e_1$, $\nu'(0)=e_\alpha$.
Leaving $x,s,t$ fixed, the first-order variation of $Z$ with respect to $h$ yields
	\begin{equation}\label{eq:transverse}
		 u_\alpha(y)-u_\alpha(z)=\kappa(s-t)x_\alpha \qquad(\alpha\ge2).
	\end{equation}
Combining \eqref{eq:transverse} with \eqref{eq:first_lengths_sum}, we obtain the first identity
	\begin{equation}\label{eq:jump}
	\begin{split}
		Du(y)-Du(z) &= [u_1(y)-u_1(z)]e_1 + \sum_{\alpha=2}^n[u_\alpha(y)-u_\alpha(z)]e_\alpha \\
			&= \Big[\frac{\omega(d)+\kappa(|y|^2+|z|^2)}2 +\frac{d\omega'(d)}2+\kappa(s^2+t^2)\Big]e_1 +\kappa(s-t)x.
	\end{split}
	\end{equation}

To derive the last two identities, we consider a different variation of $Z$ by fixing $s$, $t$, and $\nu$, and varying $x$. The first-order variation gives
	\begin{equation}
		\frac tdDu(y)+\frac sdDu(z)-Du(x) =\frac{\kappa st}{d}(y+z).
	\end{equation}
Since $d=t+s$, by rearranging we obtain
	\begin{equation}\label{eq:bary}
 	\begin{aligned}
 		Du(y)-Du(x)&=\frac sd\bigl(Du(y)-Du(z)\bigr) +\frac{\kappa st}{d}(y+z),\\
 		Du(z)-Du(x)&=-\frac td\bigl(Du(y)-Du(z)\bigr) +\frac{\kappa st}{d}(y+z).
	 \end{aligned}
	\end{equation}
\end{proof}
 
\begin{lemma} 
Under the assumption \eqref{eq:norm}, $1\leq g=f(\cdot,u,Du)\leq Q$ in $B_1$, the following upper bound holds at the middle point:
	\begin{equation}\label{eq:trace} 
 		u_{11}(x) \le 6Q\omega'(d)+\frac18\Delta u. 
	\end{equation}
\end{lemma}

\begin{proof}
Again, keeping $s$, $t$, and $\nu$ fixed and varying $x$,
the Hessian of $Z$ is given by
	\begin{equation}\label{eq:translation_hessian}
		 D_x^2Z =\frac td D^2u(y)+\frac sd D^2u(z) -D^2u(x)-\frac{2\kappa st}{d}I,
	\end{equation}
where $I$ is the $n\times n$ identity matrix. 
By maximality, $D_x^2Z\leq0$.
Contracting with the positive definite matrix $(S_2^{ij})$ and using \eqref{eq:homogeneity}, we obtain
	\begin{equation}\tag{T}\label{eq:T}
		 0\ge\frac td S_2^{ij}u_{ij}(y) +\frac sd S_2^{ij}u_{ij}(z) -2g(x)-\frac{2\kappa st}{d}\sum_iS_2^{ii}.
	\end{equation}
\stepcounter{equation}

We use Lemma~\ref{lem:directional} to estimate $S_2^{ij}u_{ij}(y)$ and $S_2^{ij}u_{ij}(z)$.
Taking $H=D^2u(x)$ and $K=D^2u(y),D^2u(z)$ in \eqref{eq:matrixlemma}, and using \eqref{eq:middle}, \eqref{eq:endpointceiling}, $1\le g\le Q$, and
$\omega'(d)>\sqrt Q$ by \eqref{eq:positive}, we can obtain
	\begin{equation}\label{eq:directional_here}
		 S_2^{ij}u_{ij}(y)\ge\frac{u_{11}(x)}{3\omega'(d)},
		 \qquad
 		S_2^{ij}u_{ij}(z)\ge\frac{u_{11}(x)}{3\omega'(d)}.
	\end{equation}
Substituting \eqref{eq:directional_here} into \eqref{eq:T}
and using $st\le d^2/4$, we then obtain
	\[ 0 \geq \frac{u_{11}(x)}{3\omega'(d)} - 2Q -\frac{\kappa d (n-1)}{2} \Delta u, \]
and thus
	\begin{equation*} 
	\begin{split}
 		u_{11}(x) &\le 6Q\omega'(d) +\frac32\kappa d\omega'(d)(n-1)\Delta u \\
 			&\le 6Q\omega'(d)+\frac18\Delta u,
	\end{split}
	\end{equation*}
where the last inequality follows from
$\beta\ge12\kappa(n-1)$ by \eqref{eq:constants}
and $\beta d\omega'(d)<1$ by \eqref{eq:riccati-m}.
\end{proof}

\begin{lemma}
Recall the choice of constants $C_0, \beta$ in \eqref{eq:constants}. We have the following estimates
	\begin{equation}\label{eq:coercivity}
 		S_2^{11}\omega'(d)^2\ge\frac{\Delta u}{3C_0^2},
		 \qquad
 		S_2^{11}\omega'(d)\ge\frac1{3C_0^2},
 		\qquad
		 \frac{S_2^{11}}d>\frac\beta{3C_0^2}.
	\end{equation}
\end{lemma}

\begin{proof}
It suffices to prove the first inequality, as the other
two follow directly from it. Indeed, since $S_2^{11}>0$,
we have $\Delta u>u_{11}(x)\ge\omega'(d)$
by \eqref{eq:middle}, which gives the second inequality.
The third then follows from $\beta d\omega'(d)<1$
by \eqref{eq:riccati-m}.

We now prove the first inequality in \eqref{eq:coercivity}.
Recall that $S_2[u]=\sum_{1\leq i<j\leq n}(u_{ii}u_{jj}-u_{ij}^2)$ (see, for example, \cite{Wa09}). 
From the equation $S_2[u]=g$, we have
	\begin{equation}\label{eq:block}
	\begin{split}
		1\le g(x)=S_2[u](x) &= u_{11}S_2^{11} + \sum_{2\leq i<j\leq n}(u_{ii}u_{jj}-u_{ij}^2) - \sum_{i=2}^nu_{1i}^2 \\
			&\le u_{11}(x)S_2^{11}+\frac12(S_2^{11})^2 \le S_2^{11}\Delta u.
	\end{split}
	\end{equation}
Since $0<S_2^{11} = \Delta u - u_{11}(x)$ 
and $u_{11}(x) > 0$ by \eqref{eq:middle},
we have $0 < S_2^{11}<\Delta u$.
Combining with \eqref{eq:block}, we obtain
	\[ 1\le\sqrt{S_2^{11}\Delta u}, \qquad\text{and }\quad S_2^{11}\le\sqrt{S_2^{11}\Delta u}. \]
Then \eqref{eq:trace} implies
	\begin{equation}\label{eq:algebra}
	\begin{aligned}
		\frac78\Delta u \le6Q\omega'(d)+S_2^{11} 
		\le(6Q+1)\omega'(d) \sqrt{S_2^{11}\Delta u}
		\le7Q\omega'(d)\sqrt{S_2^{11}\Delta u},
	\end{aligned}
	\end{equation}
where we used $\omega'(d)>1$ and $Q\ge1$.
Rearranging, we obtain
	\[ S_2^{11}\omega'(d)^2 \ge\frac{\Delta u}{64Q^2} \ge\frac{\Delta u}{3C_0^2}, \]
since $C_0=12Q$. This proves the first inequality in \eqref{eq:coercivity}.
\end{proof}

\begin{lemma}\label{lem26}
Last, by the maximality of $Z$ at $(x,y,z)$ and using a reflected variation, we have the following key estimate.
	\begin{equation}\tag{R}\label{eq:R}
	\begin{aligned}
		0\ge{}&-\sqrt Q\left[ \frac td|g(y)-g(x)|+\frac sd|g(z)-g(x)|\right] \\
			&+2S_2^{11}\left(\frac{\omega(d)}d-\omega'(d)\right) +\frac{2\kappa S_2^{11}}d|x|^2 -7\kappa d(n-1)\Delta u.
	\end{aligned}
	\end{equation}
\end{lemma}

\begin{proof}
We introduce the reflected variation as follows. 
The $n\times n$ matrix
	\[ R=I-2e_1\otimes e_1=\text{diag}(-1,1,\cdots,1)\]
denotes reflection across the hyperplane perpendicular to $e_1$. 
For simplicity, write $R=\operatorname{diag}(\varepsilon_1,\ldots,\varepsilon_n)$,
where $\varepsilon_1=-1$ and $\varepsilon_i=1$ for $i\ge2$.

For $\xi\in\mathbb R^n$ with $|\xi|$ sufficiently small, consider the reflected variation
	\begin{equation}\label{eq:reflectedpath}
		x\longmapsto x+\xi,\qquad
		y\longmapsto y+R\xi,\qquad
		z\longmapsto z+R\xi.
	\end{equation}
Then $s\longmapsto s-2\xi_1$ and $t\longmapsto t+2\xi_1$, while $d=s+t$ remains unchanged. 
The three perturbed points $(x+\xi, y+R\xi, z+R\xi)$ remain exactly collinear and lie on a new line parallel to $e_1$, since $R\xi-\xi$ is parallel to $e_1$. 
For sufficiently small $|\xi|$, the new middle point also remains between the two new endpoints. 

Along this variation, the text function $Z$ is given by
	\begin{equation}
	\begin{aligned}
		Z(\xi) &=\frac{t+2\xi_1}{d}u(y+R\xi) +\frac{s-2\xi_1}{d}u(z+R\xi) -u(x+\xi) \\
			&\qquad-\frac{(s-2\xi_1)(t+2\xi_1)}{2d} \left[ \omega(d)+\kappa\bigl(|y+R\xi|^2+|z+R\xi|^2\bigr)\right].
	\end{aligned}
	\end{equation}
By maximality, $D^2_{\xi\xi}Z(0)\leq 0$. Since $(S_2^{ij})$ is positive
definite, contracting gives 
	\[ 0\geq S_2^{ij}\partial_{\xi_i}\partial_{\xi_j}Z(0). \]
We compute the following three contributions separately.

 \emph{(I) The terms involving $u$.}
Direct differentiation gives
	\begin{equation}\label{eq:reflection_u}
	\begin{aligned}
		&S_2^{ij} \left.\partial_{\xi_i}\partial_{\xi_j} \left[ \frac{t+2\xi_1}{d}u(y+R\xi) +\frac{s-2\xi_1}{d}u(z+R\xi) -u(x+\xi) \right]\right|_{\xi=0}\\
		&\quad= \frac td\varepsilon_i\varepsilon_j S_2^{ij}u_{ij}(y) +\frac sd\varepsilon_i\varepsilon_j S_2^{ij}u_{ij}(z) -2g(x)+ \frac4d\varepsilon_iS_2^{1i} \bigl(u_i(y)-u_i(z)\bigr),
	\end{aligned}
	\end{equation}
where we used $S_2^{ij}u_{ij}(x)=2g(x)$.

\emph{(II) The term involving $\omega$.}
Since $d$ is fixed under the variation, we have
	\begin{equation}\label{eq:reflection_omega}
		-S_2^{ij} \left.\partial_{\xi_i}\partial_{\xi_j} \left[ \frac{(s-2\xi_1)(t+2\xi_1)}{2d} \omega(d) \right]\right|_{\xi=0}=\frac{4S_2^{11}}d\omega(d).	
	\end{equation}

\emph{(III) The localisation term.}
Direct computation gives
	\begin{equation}\label{eq:reflection_cutoff}
	\begin{aligned}
		&-\kappa S_2^{ij} \left.\partial_{\xi_i}\partial_{\xi_j} \left[ \frac{(s-2\xi_1)(t+2\xi_1)}{2d} \bigl(|y+R\xi|^2+|z+R\xi|^2\bigr) \right]\right|_{\xi=0}\\
		&\quad= \frac{4\kappa S_2^{11}}d \bigl(|y|^2+|z|^2\bigr) -\frac{4\kappa(s-t)}d \varepsilon_i S_2^{1i}(y_i+z_i) -\frac{2\kappa st}{d}\sum_iS_2^{ii}.
	\end{aligned}
	\end{equation}

We now combine these three calculations.
We first simplify (I) and (II). From \eqref{eq:jump},
	\begin{equation}\label{eq:reflection_jump}
	\begin{aligned}
		\frac4d\varepsilon_iS_2^{1i} \bigl(u_i(y)-u_i(z)\bigr)&=-\frac{2S_2^{11}}d \left[ \omega(d)+\kappa(|y|^2+|z|^2) +d\omega'(d)+2\kappa(s^2+t^2) \right]\\
					&\qquad+\frac{4\kappa(s-t)}d \left(S_2^{i1}x_i-2S_2^{11}x_1\right).
	\end{aligned}
	\end{equation}
In particular, combining the terms involving $\omega$ and $\omega'$ in (II) and \eqref{eq:reflection_jump} gives
	\[ -\frac{2S_2^{11}}d\omega(d) -2S_2^{11}\omega'(d) +\frac{4S_2^{11}}d\omega(d)= 2S_2^{11} \left(\frac{\omega(d)}d-\omega'(d)\right). \]
Hence, adding (I) and (II) yields
	\begin{equation}\label{eq:reflection_I_II}
	\begin{aligned}
		\text{(I)}+\text{(II)} ={}&\frac td\varepsilon_i\varepsilon_jS_2^{ij}u_{ij}(y) +\frac sd\varepsilon_i\varepsilon_jS_2^{ij}u_{ij}(z) -2g(x) +2S_2^{11} \left(\frac{\omega(d)}d-\omega'(d)\right)\\
				&\qquad-\frac{2\kappa S_2^{11}}d \left[|y|^2+|z|^2+2(s^2+t^2)\right] +\frac{4\kappa(s-t)}d \left(S_2^{i1}x_i-2S_2^{11}x_1\right).
	\end{aligned}
	\end{equation}
Note that
	\begin{equation}\label{eq-y+z-and-|y|^2+|z|^2}
	\begin{aligned}
		y+z &=2x+(s-t)e_1,\\
		|y|^2+|z|^2 &=2|x|^2+2(s-t)x_1+s^2+t^2.
	\end{aligned}
	\end{equation}
Adding (III) to \eqref{eq:reflection_I_II}, and using \eqref{eq-y+z-and-|y|^2+|z|^2}, we then obtain
	\begin{equation}\label{eq:exactreflection}
	\begin{aligned}
		0\ge{}& \frac td\varepsilon_i\varepsilon_jS_2^{ij}u_{ij}(y) +\frac sd\varepsilon_i\varepsilon_jS_2^{ij}u_{ij}(z) -2g(x) +2S_2^{11} \left(\frac{\omega(d)}d-\omega'(d)\right) -\frac{2\kappa st}{d}\sum_iS_2^{ii}\\
			&+\frac{\kappa}{d}\Big[ 4S_2^{11}|x|^2 +4(s-t)(3S_2^{11}x_1-S_2^{i1}x_i) +2S_2^{11}(s-t)^2-4S_2^{11}st \Big].
	\end{aligned}
	\end{equation}

We first estimate the expression in square brackets.
By completing the square, we have
	\[ 2S_2^{11}|x|^2 +4(s-t)(3S_2^{11}x_1-S_2^{i1}x_i) \ge -\frac{2(s-t)^2}{S_2^{11}} \sum_i(3S_2^{11}\delta_{i1}-S_2^{i1})^2. \]
Since $(S_2^{ij})$ is positive definite, we have $(S_2^{i1})^2\le S_2^{11}S_2^{ii}$ for $1\le i\le n$.
Consequently,
	\begin{equation}\label{eq:vectorbound} 
		\sum_i(3S_2^{11}\delta_{i1}-S_2^{i1})^2 =3(S_2^{11})^2+\sum_i(S_2^{i1})^2 \le S_2^{11} \left(3S_2^{11}+\sum_iS_2^{ii}\right).
	\end{equation}
Therefore, we have
	\[ 2S_2^{11}|x|^2 +4(s-t)(3S_2^{11}x_1-S_2^{i1}x_i) \ge -2(s-t)^2 \left(3S_2^{11}+\sum_iS_2^{ii}\right), \]
and hence the estimate for the expression in square brackets  
	\begin{equation}\label{eq:reserve_calc}
	\begin{aligned}
		\Big[ \cdots \Big] & \ge 2S_2^{11}|x|^2 -4S_2^{11}\bigl[(s-t)^2+st\bigr] -2(s-t)^2\sum_iS_2^{ii}\\
			& \ge 2S_2^{11}|x|^2-6d^2\sum_iS_2^{ii},
	\end{aligned}
	\end{equation}
where we used $(s-t)^2+st\le d^2$, $(s-t)^2\le d^2$, and $S_2^{11}\le\sum_iS_2^{ii}$.

Substituting \eqref{eq:reserve_calc} into \eqref{eq:exactreflection} and using $st\le d^2/4$, the sum of the last two terms in \eqref{eq:exactreflection} is bounded below by
	\begin{equation}\label{eq:reserve_cutoff}
		\frac{\kappa}{d}\Big[\cdots\Big] -\frac{2\kappa st}{d}\sum_iS_2^{ii} \ge \frac{2\kappa S_2^{11}}d|x|^2
-7\kappa d(n-1)\Delta u, 
	\end{equation}
where we used $\sum_iS_2^{ii}=(n-1)\Delta u$.

We next estimate the first three terms in \eqref{eq:exactreflection}. 
The matrix $(\varepsilon_i\varepsilon_jS_2^{ij})$ is the linearised matrix at $RD^2u(x)R$. 
Since $\sigma_2(RD^2u(x)R)=g(x)$, the mixed $\sigma_2$ inequality \eqref{eq:mixed_sigma2} gives
	\begin{equation}\label{eq:mixedsource}
	\begin{aligned}
		&\frac td\varepsilon_i\varepsilon_jS_2^{ij}u_{ij}(y) +\frac sd\varepsilon_i\varepsilon_jS_2^{ij}u_{ij}(z)-2g(x)\\
		&\quad\ge 2\sqrt{g(x)}\left[ \frac td\sqrt{g(y)}+\frac sd\sqrt{g(z)}-\sqrt{g(x)}\right]\\
		&\quad\ge-\sqrt Q\left[ \frac td|g(y)-g(x)|+\frac sd|g(z)-g(x)|\right],
	\end{aligned}
	\end{equation}
where we used $1\le g\le Q$. 

Combining \eqref{eq:exactreflection}, \eqref{eq:reserve_cutoff}, and \eqref{eq:mixedsource}, we obtain the desired estimate \eqref{eq:R}.
\end{proof}

\subsection{Dini control and the contradiction}\label{SS23}

In this subsection, we use Lemma \ref{lem26} to derive a contradiction, thereby establishing $Z\leq 0$. The desired estimate then follows from Proposition \eqref{pp2.1}, completing the proof of Theorem \ref{thm:main}. 

\begin{proof}[Proof of Theorem \ref{thm:main}]
We first consider the terms in \eqref{eq:R} involving $g$.
By \eqref{eq:hypothesis}, $|Du|\le M$, and monotonicity of $\delta_x$, we have
	\begin{equation}\label{error-p}
	\begin{aligned}
		& \frac td|g(y)-g(x)|+\frac sd|g(z)-g(x)| \\
		\le \ & \delta_x(d)+L_uMd +L_p\left(\frac td|Du(y)-Du(x)|+\frac sd|Du(z)-Du(x)|\right).
	\end{aligned}
	\end{equation}
At this point, we use Lemma \ref{lem23} to control the gradient difference. 
By \eqref{eq:bary}, we have
	\begin{equation}\label{eq:weightedp}
		 \frac td|Du(y)-Du(x)|+\frac sd|Du(z)-Du(x)| \le \frac{2st}{d^2}|Du(y)-Du(z)| +\frac{\kappa st}{d}|y+z|.
	\end{equation}
To estimate the right side, we use \eqref{eq:jump} and \eqref{eq-y+z-and-|y|^2+|z|^2}. 
Since $\frac{st}{d^2} \le \frac14$, $|s-t|\le d$, $s^2+t^2\le d^2$, and
$d\omega'(d)\le\omega(d)$, we obtain
	\begin{equation}\label{eq:gradcontrol}
	\begin{aligned}
		 \frac{2st}{d^2}|Du(y)-Du(z)|+\frac{\kappa st}{d}|y+z| &\le\frac{\omega(d)}2+\frac{\kappa}2|x|^2 +\frac{3\kappa}2d|x|+\kappa d^2\\
 			&\le\frac{\omega(d)}2+\kappa|x|^2+3\kappa d^2,
	\end{aligned}
	\end{equation}
where, in the last step, we also used the Cauchy-Schwarz inequality $\frac32d|x|\le\frac12|x|^2+\frac98d^2$.
Substituting \eqref{eq:gradcontrol} into \eqref{eq:weightedp}, then into \eqref{error-p}, we obtain
	\begin{equation}\label{eq:sourceerror}
	\begin{aligned}
		\frac td|g(y)-g(x)|+\frac sd|g(z)-g(x)| &\le\delta_x(d)+L_uMd +L_p\left(\frac{\omega(d)}2+\kappa|x|^2+3\kappa d^2\right)\\
 				&=\delta(d)+L_p\left(\frac{\omega(d)}2+\kappa|x|^2\right),
	\end{aligned}
	\end{equation}
where $\delta$ is the combined Dini modulus defined in \eqref{eq:delta}.

Next, we estimate $\frac{\omega(d)}{d}-\omega'(d)$ in \eqref{eq:R}.
Using integration by parts, \eqref{eq:riccati}, and the definition of $b$ in \eqref{eq:b}, we have
	\begin{equation}\label{eq:gain_identity}
	\begin{aligned}
		\frac{\omega(d)}d-\omega'(d) &=\frac1d\left(\int_0^d\omega'(r)\,dr -d\omega'(d)\right) =\frac1d\int_0^d r[-\omega''(r)]\,dr\\
		&=\frac{\beta}{d}\int_0^d r\omega'(r)^2\,dr +\frac{6C_0^2\sqrt Q}{d} \int_0^d\delta(2r)\omega'(r)\,dr.
	\end{aligned}
	\end{equation}
For the second integral involving the Dini term, since $\delta(2r)\ge\delta(d)$ for $\frac12d \le r\le d$
and $\omega'$ is positive and decreasing, we have
	\[ \frac{6C_0^2\sqrt Q}{d}\int_0^d\delta(2r)\omega'(r)\,dr \ge\frac{6C_0^2\sqrt Q}{d} \int_{d/2}^d\delta(2r)\omega'(r)\,dr \ge3C_0^2\sqrt Q\,\delta(d)\omega'(d). \]	
Hence, from \eqref{eq:gain_identity}, and again, since $\omega'$ is decreasing, we have
	\begin{equation}\label{eq:gain} 
		\frac{\omega(d)}d-\omega'(d) \ge\frac{\beta d}{2}\omega'(d)^2+3C_0^2\sqrt Q\,\delta(d)\omega'(d).
	\end{equation}
Together with \eqref{eq:coercivity}, \eqref{eq:gain} implies
	\begin{equation}\label{eq:firstgain}
 		S_2^{11}\left(\frac{\omega(d)}d-\omega'(d)\right) \ge\frac{\beta d}{6C_0^2}\Delta u +\sqrt Q\,\delta(d).
	\end{equation}

We will also use the following estimate
	\begin{equation}\label{eq:secondgain}
 		S_2^{11}\left(\frac{\omega(d)}d-\omega'(d)\right) \ge\frac{\beta}{12C_0^2}\omega(d).
	\end{equation}
Indeed, if $\frac{\omega(d)}{d}\le2\omega'(d)$, then by \eqref{eq:gain} and the second inequality in
\eqref{eq:coercivity} we have
	\begin{align*} 
		S_2^{11}\left(\frac{\omega(d)}d-\omega'(d)\right) &\ge \frac{\beta d}{2}\omega'(d)\left(S_2^{11}\omega'(d)\right) \\
			&\geq \frac{\beta d}{6C_0^2} \omega'(d) \ge\frac\beta{12C_0^2}\omega(d). 
	\end{align*}
If $\frac{\omega(d)}{d}>2\omega'(d)$, then $\frac{\omega(d)}{d}-\omega'(d)>\frac{\omega(d)}{2d}$ and the last
inequality in \eqref{eq:coercivity} imply \eqref{eq:secondgain}.

Substituting \eqref{eq:sourceerror} into \eqref{eq:R}, 
applying \eqref{eq:firstgain} and \eqref{eq:secondgain} to the two copies of $S_2^{11}\big(\frac{\omega(d)}{d}-\omega'(d)\big)$ in \eqref{eq:R},
and using the last inequality in \eqref{eq:coercivity} to estimate the coefficient $\frac{S_2^{11}}{d}$ of $|x|^2$ in \eqref{eq:R}, we obtain
	\begin{equation}\label{eq:contradiction} 
		 0\ge d\Delta u \left[\frac\beta{6C_0^2}-7\kappa(n-1)\right]
 			+\omega(d) \left[\frac\beta{12C_0^2}-\frac{\sqrt Q L_p}{2}\right]
			 +\kappa|x|^2 \left[\frac{2\beta}{3C_0^2}-\sqrt Q L_p\right]. 
	\end{equation}
The choice of constants in \eqref{eq:constants} ensures that all three bracketed coefficients are strictly positive. 
Since $d>0$, $\omega(d)>0$, and $\Delta u>0$, the right-hand side is strictly positive, which gives the desired contradiction.
\end{proof}

\begin{remark}
The notation $Z(x,y,z)$ denotes the three-point test function on collinear triples, while $Z(x,\nu,s,t)$ is its coordinate representation. 
The admissible domain has $2n+1$ degrees of freedom, corresponding to $x\in\mathbb R^n$, $\nu\in\mathbb S^{n-1}$, and $s,t>0$. 
In particular, the direction is free to vary in \eqref{eq:transverse}; it is not fixed before taking the maximum.
The argument never differentiates the equation.
Thus the a prior estimate applies directly to $C^2$ solutions under the stated assumption that $f$ is Dini continuous in the spatial variable and Lipschitz continuous in the solution and gradient variables. 
\end{remark}

\section{Interior \texorpdfstring{$C^2$}{C2} estimate for the Monge-Amp\`ere equation}

In this section, we prove Theorem~\ref{thm:ma} using a three-point maximum principle. The argument is localised on sublevel sets of $u$, whose definition we first recall. 

\begin{definition}\label{def:ma_section}
Let $u\in C^2(\Omega)$ be strictly convex. 
For $x_0\in\Omega$ and $h>0$, the sublevel set of $u$ at $x_0$ with height $h$ is defined by
	\begin{equation}\label{eq:ma_section}
 		S_h^0(x_0)=\left\{x\in\Omega : u(x)<u(x_0)+Du(x_0)\cdot(x-x_0)+h\right\}.
	\end{equation}
For $x,y\in\Omega$, we define
	\begin{equation}\label{eq:ma_rho}
 		\rho_x(y)=u(y)-u(x)-Du(x)\cdot(y-x).
	\end{equation}
Then $S_h^0(x_0)=\{x\in\Omega:\rho_{x_0}(x)<h\}$.
\end{definition}

Let $\Omega'\Subset\Omega''\Subset\Omega$. By the strict convexity of $u$, we may choose $H>0$, uniformly for $x_0\in\Omega'$, such that
	\begin{equation}\label{eq:ma_section_compact}
		S_H^0(x_0)\Subset\Omega'' \qquad\forall\,x_0\in\Omega'.
	\end{equation}
In the normalised zero-boundary setting of \eqref{eq:ma_PDE}, the quantitative strict convexity estimate (see, for example, \cite[\S2.3]{LW15}) ensures that $H$ can be chosen to depend only on $n,\lambda,\Lambda,\Omega',\Omega''$. The resulting estimate is uniform in $x_0$.

As in Section 2, the proof is organised into three parts. 
In \S\ref{SS31}, we construct a three-point test function $Z$ and show that $Z\leq0$ implies the desired second derivative estimate. 
In \S\ref{SS32}, we derive estimates by considering different variations at a positive interior maximum of $Z$. 
In \S\ref{SS33}, we combine these estimates to obtain a contradiction and complete the proof.

\subsection{Three-point function}\label{SS31}

Fix $x_0\in\Omega'$ and write $\rho=\rho_{x_0}$. 
Consider an arbitrary direction $\nu\in\mathbb S^{n-1}$ and three
collinear points $z,x,y$, with $x$ between the endpoints. Write
	\begin{equation}\label{eq:ma_chord}
		y=x+s\nu,\qquad z=x-t\nu,\qquad s,t>0,\qquad d=s+t\le s_0
	\end{equation}
for a small constant $s_0$ to be determined.
On the set of such triples with $y,z\in\overline{S_H^0(x_0)}$, allowing $\nu$ to vary,
define the test function
	\begin{equation}\label{eq:ma_Z}
		Z(x,\nu,s,t) =\frac td u(y)+\frac sd u(z)-u(x) -\frac{st}{2d}\bigl[\omega(d)+K(\rho(y)+\rho(z))\bigr].
	\end{equation}
Similar to \eqref{eq:Z}, this test function also has three ingredients: a weighted finite difference of $u$; a distance profile function $\omega$; and a localisation term $K(\rho(y)+\rho(z))$, where the quadratic function in \eqref{eq:Z} is replaced by $\rho$. 
The profile $\omega$ and the constant $K$ will be chosen below.

Set
	\begin{equation}\label{eq:ma_norm}
 		M=\max\left\{1,|Du|_{L^\infty(\Omega'')}\right\}.
	\end{equation} 
Choose the constants 
	\begin{equation}\label{eq:ma_constants}
	\begin{aligned}
 		K&=\frac{4M+1}{H},\qquad C_0=9e^{2L_pM}\left(\frac{3n}{n-1}\right)^{n-1},\\
 		\beta&=C_0\bigl((n+8)K+4L_p+2\bigr),\qquad T_0=\max\{1,2KM\}.
	\end{aligned}
	\end{equation}

Write $g(x)=G(x,u(x),Du(x))=\log\det D^2u(x)$, where $G$ satisfies \eqref{eq:ma_hypothesis}. 
Define a combined Dini modulus by
	\begin{equation}\label{eq:ma_delta}
 		\delta_0(r)=\delta_x(r) +\left(L_uM+\frac32L_pKM\right)r.
	\end{equation}
The function $\delta_0$ is again a continuous, nondecreasing Dini modulus.
Choose $s_0>0$ sufficiently small such that
	\begin{equation}\label{eq:ma_s0}
		 Ks_0\le\frac12,\quad 2s_0\le\operatorname{diam}(\Omega''),\quad  \delta_x(s_0)+L_uMs_0\le1.
	\end{equation}
This choice can be made uniformly, since $\delta_x$ is nondecreasing and has finite Dini integral.

Define, for $0<r\le s_0$,
	\begin{equation}\label{eq:ma_b}
	\begin{aligned}
		 b(r)&=2C_0\frac{\delta_0(2r)}r,\\
 		b_1(r)&=\int_0^r b(\tau)\,d\tau,\\
 		b_2(r)&=\int_0^r e^{-b_1(\tau)}\,d\tau,
	\end{aligned}
	\end{equation}
with $b_1(0)=b_2(0)=0$. 
We can then define the concave, profile function 
	\begin{equation}\label{eq:ma_omega}
		 \omega(r)=\frac1\beta\log\bigl(1+\beta A b_2(r)\bigr), \qquad 0\le r\le s_0,
	\end{equation}
where the constant $A$ is chosen to be
	\begin{equation}\label{eq:ma_A}
 		A=e^{4\beta M+b_1(s_0)} \left(T_0+\frac{2}{\beta s_0}\right).
	\end{equation}

It is straightforward to check that the profile $\omega$ satisfies 
	\begin{equation}\label{eq:ma_riccati}
	\begin{aligned}
 		\omega'(r)&=Ae^{-b_1(r)-\beta\omega(r)}>0,\\
 		\omega''(r)&=-\beta\omega'(r)^2-b(r)\omega'(r)<0 \qquad(r>0),\\
 		\beta r\omega'(r)&<1.
	\end{aligned}
	\end{equation}
The last assertion follows from $b_2(r)\ge r e^{-b_1(r)}$, by the same argument used in \eqref{eq:riccati-m}.
Moreover, the choice of $A$ in \eqref{eq:ma_A} implies
	\begin{equation}\label{eq:ma_height}
		\omega(s_0)>4M.
	\end{equation}
And, whenever $\omega(r)<4M$, we have
	\begin{equation}\label{eq:ma_slope}
		\omega'(r)>Ae^{-b_1(s_0)-4\beta M}>T_0 \ge 1.
	\end{equation}
These two estimates follow by the same arguments as those for \eqref{eq:profile_height} and \eqref{eq:profile_slope}, respectively.

%
We first explain how the desired estimate follows once $Z\le0$ is established. 
\begin{proposition}\label{prop31}
Let $x=x_0$, $y=x_0+r\nu$, $z=x_0-r\nu$. If $Z\leq 0$ for any $\nu\in\mathbb{S}^{n-1}$, $r\in(0,s_0/2)$, then
	\[ |D^2u(x_0)|\le A. \]
\end{proposition}

\begin{proof}
Let $s=t=r$. As $r\downarrow0$,
	\begin{equation}\label{eq:ma_diagonal}
	\begin{aligned}
		\frac{u(x_0+r\nu)+u(x_0-r\nu)}2-u(x_0) &=\frac{r^2}{2}u_{\nu\nu}(x_0)+o(r^2),\\
		\frac r4\bigl[\omega(2r)+K(\rho(y)+\rho(z))\bigr] &=\frac{Ar^2}{2}+o(r^2).
	\end{aligned}
	\end{equation}
Thus $Z\le0$ implies $u_{\nu\nu}(x_0)\le A$ for every unit vector $\nu$. 
Since $D^2u>0$, this yields $|D^2u(x_0)|\le A$. 
\end{proof}

Note that the bound in Proposition \ref{prop31} is uniform for $x_0\in\Omega'$, which further gives the desired estimate \eqref{eq:ma_conclusion}.
It therefore suffices to prove $Z\le0$ in the admissible domain $\mathcal{D}$, consisting of all collinear triples $(x,y,z)$ satisfying \eqref{eq:ma_chord}, with $y,z\in\overline{S_H^0(x_0)}$ and $d=|y-z|\leq s_0$.

We first rule out a positive maximum on the boundary.
Since $|Du|\le M$, we have
	\[ \frac td u(y)+\frac sd u(z)-u(x) = \frac td \big(u(y)-u(x)\big) + \frac sd \big(u(z)-u(x)\big) \le\frac{2Mst}{d}, \]
and thus
	\begin{equation}\label{eq:ma_boundary}
		Z\le\frac{st}{2d} \bigl[4M-\omega(d)-K(\rho(y)+\rho(z))\bigr].
	\end{equation}
If either one of $y, z$ lies on $\partial S_H^0(x_0)$, then $K(\rho(y)+\rho(z))\ge KH=4M+1$, and $Z<0$.
If $d=s_0$, then $\omega(s_0)>4M$ by \eqref{eq:ma_height}, so again $Z<0$.
On $s=0$ or $t=0$, the continuous extension of $Z$ is zero.

Therefore, we conclude that
\begin{proposition}\label{prop32}
If $Z$ takes a positive value, it attains a positive maximum in the interior of the admissible domain $\mathcal{D}$ at a triple of distinct collinear points $(x,y,z)$ for some $s,t>0$ and $d<s_0$.
At this maximum, \eqref{eq:ma_boundary} gives
	\begin{equation}\label{eq:ma_positive}
		\omega(d)+K(\rho(y)+\rho(z))<4M.
	\end{equation}
Since $\rho\ge0$, we have $\omega(d)<4M$. 
It follows from \eqref{eq:ma_slope} that
	\begin{equation}\label{eq:ma_contact_slope}
		\omega'(d)>T_0 \ge 1.
	\end{equation}
\end{proposition}

\subsection{Estimates by variations}\label{SS32}

In this subsection, we shall derive estimates from different variations of $Z$ at the maximum triple $(x,y,z)$.
By a rotation of coordinates, we may assume that
	\begin{equation*}
 		\nu=e_1,\qquad y=x+se_1,\qquad z=x-te_1.
	\end{equation*}
For simplicity, denote
	\begin{equation}\label{eq:ma_F}
		F=(F^{ij})=(D^2u(x))^{-1}.
	\end{equation}
In particular, $F$ is positive definite and $F^{ij}u_{ij}(x)=n$.

\begin{lemma}
The following upper bounds hold at two endpoints:
	\begin{equation}\label{eq:ma_endpoint_ceiling}
		u_{11}(y),\ u_{11}(z)  < 3\omega'(d).
	\end{equation}
\end{lemma}

\begin{proof}
To derive \eqref{eq:ma_endpoint_ceiling}, we vary the endpoint $y$ along $e_1$ with $x,z$ fixed, and vary the endpoint $z$ along $e_1$ with $x,y$ fixed, respectively.
In other words, vary $s$ and $t$ separately.
From $Z_s=Z_t=0$, we obtain
	\begin{equation}\label{eq:ma_lengths}
	\begin{aligned}
		0 &= u_1(y)-\frac{u(y)-u(z)}d -\frac{t}{2d}[\omega(d)+K(\rho(y)+\rho(z))] -\frac s2[\omega'(d)+K\rho_1(y)],\\
		0 &= \frac{u(y)-u(z)}d-u_1(z) -\frac{s}{2d}[\omega(d)+K(\rho(y)+\rho(z))] -\frac t2[\omega'(d)-K\rho_1(z)].
	\end{aligned}
	\end{equation}
For the second derivatives $Z_{ss},Z_{tt}\le0$, a straightforward computation yields
	\begin{equation}\label{eq:ma_endpoint_second}
	\begin{aligned}
		0 &\geq (1-Ks/2)u_{11}(y)-\omega'(d) -\frac s2\omega''(d)-K\rho_1(y),\\
		0 &\geq (1-Kt/2)u_{11}(z)-\omega'(d) -\frac t2\omega''(d)+K\rho_1(z).
	\end{aligned}
	\end{equation}
Since $|D\rho|\le2M$ by \eqref{eq:ma_norm}, $Ks,Kt\le1/2$ by \eqref{eq:ma_s0}, and $\omega''<0$ by \eqref{eq:ma_riccati}, the bound $\omega'(d)>T_0\ge2KM$ from \eqref{eq:ma_contact_slope} yields
$u_{11}(y),\ u_{11}(z)\le(8/3)\omega'(d) < 3\omega'(d).$
\end{proof}

\begin{lemma}\label{lem32}
We have the following identities for the gradient differences:
	\begin{align*}
		Du(y)-Du(z) &= \frac12\bigl[\omega(d)+K(\rho(y)+\rho(z)) +d\omega'(d)\bigr]e_1 +\frac K2\bigl[sD\rho(y)-tD\rho(z)\bigr] \\
		Du(y)-Du(x) &= \frac sd[Du(y)-Du(z)] +\frac{Kst}{2d}[D\rho(y)+D\rho(z)],\\
		Du(z)-Du(x) &=-\frac td[Du(y)-Du(z)] +\frac{Kst}{2d}[D\rho(y)+D\rho(z)].
	\end{align*}
\end{lemma}

\begin{proof}
Adding the two identities in \eqref{eq:ma_lengths}, we first get
	\begin{equation}\label{eq:ma_lengths-1} 
		u_1(y)-u_1(z) = \frac12[\omega(d)+K(\rho(y)+\rho(z))]+\frac d2\omega'(d)+\frac K2[s\rho_1(y)-t\rho_1(z)].
	\end{equation}

For $\alpha\ge2$, vary $\nu$ along a smooth curve on $\mathbb{S}^{n-1}$ satisfying $\nu(0)=e_1$ and $\nu'(0)=e_\alpha$, while keeping $x,s,t$ fixed. 
The first-order variation of $Z$ yields
	\begin{equation}\label{eq:ma_transverse}
		 u_\alpha(y)-u_\alpha(z)  =\frac K2\bigl[s\rho_\alpha(y)-t\rho_\alpha(z)\bigr] \qquad(\alpha\ge2).
	\end{equation}
Together with \eqref{eq:ma_lengths-1}, we obtain
	\begin{equation}\label{eq:ma_gradient_jump} 
		Du(y)-Du(z) = \frac12\bigl[\omega(d)+K(\rho(y)+\rho(z)) +d\omega'(d)\bigr]e_1 +\frac K2\bigl[sD\rho(y)-tD\rho(z)\bigr]. 
	\end{equation} 

Considering a different variation of $Z$ by keeping $s,t,\nu$ fixed and varying $x$, we obtain
	\begin{equation}
		 \frac tdDu(y)+\frac sdDu(z)-Du(x) =\frac{Kst}{2d}[D\rho(y)+D\rho(z)].
	\end{equation}
Since $d=t+s$, by rearranging we obtain
	\begin{equation}\label{eq:ma_gradient_average}
	\begin{aligned}
		Du(y)-Du(x) &=\frac sd[Du(y)-Du(z)] +\frac{Kst}{2d}[D\rho(y)+D\rho(z)], \\
		Du(z)-Du(x) &=-\frac td[Du(y)-Du(z)] +\frac{Kst}{2d}[D\rho(y)+D\rho(z)].
	\end{aligned}
	\end{equation}
\end{proof}

\begin{lemma}
The following upper bound holds at the middle point:
	\begin{equation}\label{eq:ma_metric}
		u_{11}(x)\le C_0\omega'(d).
	\end{equation}
\end{lemma}

\begin{proof}
Again, keeping $s,t,\nu$ fixed and varying $x$,  by maximality, the Hessian of $Z$ satisfies
	\begin{equation}\label{eq:ma_translation_hessian}
		0\geq D_x^2Z =\frac td\left(1-\frac{Ks}{2}\right)D^2u(y) +\frac sd\left(1-\frac{Kt}{2}\right)D^2u(z) -D^2u(x).
	\end{equation}
Contracting with $F>0$, we obtain
	\begin{equation}\tag{T}\label{eq:ma_T}
		0\ge \frac td\left(1-\frac{Ks}{2}\right)F^{ij}u_{ij}(y) +\frac sd\left(1-\frac{Kt}{2}\right)F^{ij}u_{ij}(z)-n.
	\end{equation}
\stepcounter{equation}

Since $Ks,Kt\le1/2$ by \eqref{eq:ma_s0} and $\max\{\frac sd, \frac td\}\ge1/2$, both coefficients
in \eqref{eq:ma_T} are positive, and at least one is no less than $3/8$. 
Choose the corresponding endpoint $w=y$ or $z$. 
From \eqref{eq:ma_T}, we have
	\begin{equation}\label{eq:ma_weighted_bound}
		F^{ij}u_{ij}(w)\le\frac{8n}{3}<3n.
	\end{equation}
The determinant comparison of Lemma~\ref{lem:ma_directional}, with $X = D^2 u(x)$, $W = D^2 u(w)$, and $\nu = e_1$, then implies
	\begin{equation}\label{eq:ma_determinant_bound}
		\frac{u_{11}(x)}{u_{11}(w)} \le e^{g(x)-g(w)} \left(\frac{F^{ij}u_{ij}(w)}{n-1}\right)^{n-1}.
	\end{equation}
By \eqref{eq:ma_hypothesis}, \eqref{eq:ma_norm}, and \eqref{eq:ma_s0}, we have
	$$g(x)-g(w)\le\delta_x(d)+L_uMd+2L_pM\le1+2L_pM.$$
Also, $u_{11}(w)\le3\omega'(d)$ by \eqref{eq:ma_endpoint_ceiling}.
Combining \eqref{eq:ma_weighted_bound} and \eqref{eq:ma_determinant_bound} with the choice of $C_0$ in \eqref{eq:ma_constants}, we thus obtain
	\begin{equation*} 
		u_{11}(x)\le C_0\omega'(d).
	\end{equation*}
Note that $u_{11}(x)>0$ since $D^2u(x)$ is positive definite.
\end{proof}

We also derive a lower bound for the right-hand side of (T) in an alternative way.
The arithmetic-geometric mean inequality gives
	\begin{equation}\label{eq:ma_mixed_det}
	\begin{aligned}
		F^{ij}u_{ij}(y)&\ge n e^{[g(y)-g(x)]/n},\\
		F^{ij}u_{ij}(z)&\ge n e^{[g(z)-g(x)]/n}.
	\end{aligned}
	\end{equation}
Using $e^a\ge1+a$ and $st\le d^2/4$, we obtain
	\begin{equation}\label{eq:ma_source}
	\begin{aligned}
		&\frac td\left(1-\frac{Ks}{2}\right)F^{ij}u_{ij}(y) +\frac sd\left(1-\frac{Kt}{2}\right)F^{ij}u_{ij}(z)-n\\
		&\quad\ge-\frac td|g(y)-g(x)|-\frac sd|g(z)-g(x)| -\frac{nKd}{4}.
	\end{aligned}
	\end{equation}

\begin{lemma}\label{lem34}
Last, by the maximality of $Z$ at $(x,y,z)$ and using an oblique reflected variation, we have the following key estimate.
	\begin{equation}\tag{R}\label{eq:ma_R}
	\begin{aligned}
		0\ge{}& -\frac td|g(y)-g(x)|-\frac sd|g(z)-g(x)| -\frac{nKd}{4}\\
			&+\frac{2}{u_{11}(x)} \left(\frac{\omega(d)}d-\omega'(d)\right) +\frac{2K}{d\,u_{11}(x)} [\rho(y)+\rho(z)+s\rho_1(z)-t\rho_1(y)].
	\end{aligned}
	\end{equation} 
\end{lemma}

\begin{proof}
We introduce the \emph{oblique} reflected variation as follows.
Define the $n\times n$ matrix
	\begin{equation}\label{eq:ma_reflection}
		R=I-\frac{2}{u_{11}(x)} e_1\otimes D^2u(x)e_1.
	\end{equation}
A direct calculation yields
	\[ R\xi = \xi - \frac{2(D^2u(x) e_1)\cdot\xi}{u_{11}(x)}e_1 \qquad\forall\,\xi\in\mathbb{R}^n. \]
This shows that $R$ changes only the first coordinate of a vector $\xi$. 
All coordinates perpendicular to $e_1$ remain unchanged.
$Rv=v$ only if $(D^2u(x) e_1)\cdot v=0$. In particular, $R^2=I$ and
	\begin{equation}\label{eq:ma_R_properties}
		Re_1=-e_1,\qquad RFR^T=F.
	\end{equation}

For $\xi\in\mathbb R^n$ with $|\xi|$ sufficiently small, take the reflected variation
$x\mapsto x+\xi$, $y\mapsto y+R\xi$, and $z\mapsto z+R\xi$.
Consequently,
	\begin{equation}\label{eq:ma_reflected_path}
	\begin{aligned}
		s\mapsto s-\frac{2}{u_{11}(x)}D^2u(x)e_1\cdot\xi,\quad t\mapsto t+\frac{2}{u_{11}(x)}D^2u(x)e_1\cdot\xi,
	\end{aligned}
	\end{equation} 
while $d=s+t$ remains unchanged. 
Along this variation, the test function $Z$ takes the form
	\begin{equation*}\label{eq:ma_reflected_Z}
	\begin{aligned}
		Z&=\frac{1}{d}\left(t+\frac{2D^2u(x)e_1\cdot\xi}{u_{11}(x)}\right)u(y+R\xi) +\frac{1}{d}\left(s-\frac{2D^2u(x)e_1\cdot\xi}{u_{11}(x)}\right) u(z+R\xi) -u(x+\xi)\\
		&\quad-\frac{1}{2d}\left(s-\frac{2D^2u(x)e_1\cdot\xi}{u_{11}(x)}\right)\left(t+\frac{2D^2u(x)e_1\cdot\xi}{u_{11}(x)}\right)\big[\omega(d)+K\bigl(\rho(y+R\xi)+\rho(z+R\xi)\bigr)\big].
	\end{aligned}
	\end{equation*}
By maximality, the Hessian of $Z$ with respect to $\xi$ is nonpositive at $\xi=0$.
Contracting it with $F$ gives 
	\[ 0 \geq F^{ij}\partial^2_{\xi_i\xi_j}Z(0). \]
We compute the following three contributions separately.

\emph{(I) The terms involving $u$.}
Using \eqref{eq:ma_R_properties} and
$F^{ij}u_{ij}(x)=n$, we obtain
	\begin{equation*}\label{eq:ma_reflection_u} 
		\mathrm{(I)} = \frac td F^{ij}u_{ij}(y) +\frac sd F^{ij}u_{ij}(z)-n -\frac{4}{d\,u_{11}(x)}[u_1(y)-u_1(z)].
	\end{equation*}

\emph{(II) The term involving $\omega$.}
Since $d$ remains fixed, we have
	\begin{equation*}\label{eq:ma_reflection_omega}
		\mathrm{(II)}=\frac{4\omega(d)}{d\,u_{11}(x)}.
	\end{equation*}

\emph{(III) The localisation term.}
Using $D^2\rho=D^2u$, we obtain
	\begin{equation*}\label{eq:ma_reflection_rho}	
		\mathrm{(III)} = -\frac{Kst}{2d}F^{ij}[u_{ij}(y)+u_{ij}(z)] +\frac{4K}{d\,u_{11}(x)}[\rho(y)+\rho(z)] +\frac{2K(s-t)}{d\,u_{11}(x)} [\rho_1(y)+\rho_1(z)].
	\end{equation*}

Adding (I)--(III) and substituting \eqref{eq:ma_lengths-1}, we then obtain
	\begin{equation*}
	\begin{aligned}
		0\ge{}& \frac td\left(1-\frac{Ks}{2}\right)F^{ij}u_{ij}(y)+\frac sd\left(1-\frac{Kt}{2}\right)F^{ij}u_{ij}(z)-n\\
			&+\frac{2}{u_{11}(x)}\left(\frac{\omega(d)}d-\omega'(d)\right) +\frac{2K}{d\,u_{11}(x)} [\rho(y)+\rho(z)+s\rho_1(z)-t\rho_1(y)].
	\end{aligned}
	\end{equation*}
Therefore, substituting \eqref{eq:ma_source} into the above inequality, we obtain the desired estimate \eqref{eq:ma_R}.	
\end{proof}

\subsection{Dini control and the contradiction}\label{SS33}

In this subsection, we derive a contradiction using Lemma \ref{lem34}, thereby establishing $Z\leq 0$. The desired estimate then follows from Proposition \ref{prop31}, completing the proof of Theorem \ref{thm:ma}.

\begin{proof}[Proof of Theorem~\ref{thm:ma}] 
We now estimate the terms in \eqref{eq:ma_R}.
Set
	\begin{equation}\label{eq:ma_w}
		 w=x+(s-t)e_1.
	\end{equation}
Then $w=y-te_1=z+se_1$ lies on the segment joining $z$ and $y$.

We first estimate the terms involving $\rho$ in \eqref{eq:ma_R}.
By \eqref{eq:ma_lengths-1} and $|D\rho|\le2M$, we have
	\begin{equation}\label{eq:ma_contact_direction}
	\begin{aligned}
		\rho_1(y)-\rho_1(z) &=\frac12[\omega(d)+K(\rho(y)+\rho(z))] +\frac d2\omega'(d) +\frac K2[s\rho_1(y)-t\rho_1(z)] \\
				&\le\frac12[\omega(d)+K(\rho(y)+\rho(z))] +\frac d2\omega'(d)+KMd.
	\end{aligned}
	\end{equation}
By convexity of $\rho$, we obtain
	\begin{equation}\label{eq:ma_taylor}
	\begin{aligned}
		\rho(y)-t\rho_1(y) &\ge\rho(w)-t[\rho_1(y)-\rho_1(z)],\\
		\rho(z)+s\rho_1(z) &\ge\rho(w)-s[\rho_1(y)-\rho_1(z)].
	\end{aligned}
	\end{equation}
Moreover, $|D\rho|\le2M$ gives $\rho(y)+\rho(z)\le2\rho(w)+2Md$.
Adding the two inequalities in \eqref{eq:ma_taylor} and using \eqref{eq:ma_contact_direction}, we obtain
	\begin{equation}\label{eq:ma_sectionerror}
	\begin{aligned}
		&\rho(y)+\rho(z)+s\rho_1(z)-t\rho_1(y)\\
		&\quad\ge2\rho(w)-\frac d2[\omega(d)+K(\rho(y)+\rho(z))] -\frac{d^2}{2}\omega'(d)-KMd^2\\
		&\quad\ge(2-Kd)\rho(w)-\frac d2\omega(d) -\frac{d^2}{2}\omega'(d)-2KMd^2.
	\end{aligned}
	\end{equation}

We next consider the terms in \eqref{eq:ma_R} involving $g$.
Here we use Lemma \ref{lem32} to control the gradient difference.  
By \eqref{eq:ma_gradient_average}, we have
	\[ \frac td|Du(y)-Du(x)|+\frac sd|Du(z)-Du(x)| \le\frac{2st}{d^2}|Du(y)-Du(z)| +\frac{Kst}{2d}|D\rho(y)+D\rho(z)|. \]
Using \eqref{eq:ma_gradient_jump}, $st\le d^2/4$, and $|D\rho|\le2M$, we obtain
	\begin{equation*} 
	\begin{aligned}
		\frac td|Du(y)-Du(x)|+\frac sd|Du(z)-Du(x)| &\le\frac14[\omega(d)+K(\rho(y)+\rho(z))] +\frac d4\omega'(d)+KMd \\
				&\le\frac{\omega(d)}2+\frac K2\rho(w) +\frac32KMd,
	\end{aligned}
	\end{equation*}
where the last inequality uses $d\omega'(d)\le\omega(d)$ and $\rho(y)+\rho(z)\le2\rho(w)+2Md$.
By \eqref{eq:ma_hypothesis}, $|Du|\le M$, and \eqref{eq:ma_delta}, it follows that
	\begin{equation}\label{eq:ma_sourceerror} 
		\frac td|g(y)-g(x)|+\frac sd|g(z)-g(x)|\le\delta_0(d)+\frac{L_p}{2}\omega(d) +\frac{L_pK}{2}\rho(w).
	\end{equation}

Substituting \eqref{eq:ma_sectionerror} and \eqref{eq:ma_sourceerror} into \eqref{eq:ma_R}, we obtain
	\begin{equation}\label{eq:ma_closing}
	\begin{aligned}
		0\ge{}&-\delta_0(d)-\frac{L_p}{2}\omega(d) -\frac{L_pK}{2}\rho(w) -\frac{nKd}{4}\\
			&+\frac{1}{u_{11}(x)}\left[ (2-Kd)\left(\frac{\omega(d)}d-\omega'(d)\right) -2Kd\omega'(d)-4K^2Md\right] \\
			&+\frac{2K \beta(2-Kd)}{C_0}\rho(w),
	\end{aligned}
	\end{equation}
where we used \eqref{eq:ma_metric} and the last inequality in \eqref{eq:ma_riccati} to estimate $[d u_{11}(x)]^{-1}\geq \beta/C_0$ in the coefficient of $\rho(w)$ in the last line.
Since $Kd\le1/2$ and $\rho(w)\ge0$, \eqref{eq:ma_closing} implies
	\begin{equation}\label{eq:ma_closing_reserve}
	\begin{aligned}
		0\ge{}&-\delta_0(d)-\frac{L_p}{2}\omega(d) -\frac{L_pK}{2}\rho(w) -\frac{nKd}{4}\\
			&+\frac1{u_{11}(x)}\left[\frac{\omega(d)}d-\omega'(d)-2Kd\omega'(d)-4K^2Md\right]\\
			&+\frac{1}{2u_{11}(x)}\left(\frac{\omega(d)}d-\omega'(d)\right)	+\frac{3K\beta}{C_0}\rho(w).
	\end{aligned}
	\end{equation}

We first estimate $\frac{\omega(d)}{d}-\omega'(d)$.
Using integration by parts, the second identity in \eqref{eq:ma_riccati},
and the definition of $b$ in \eqref{eq:ma_b}, similar to \eqref{eq:gain_identity}, we have
	\begin{equation}\label{eq:ma_gain_identity}
	\begin{aligned}
		\frac{\omega(d)}d-\omega'(d) &=\frac1d\left(\int_0^d\omega'(r)\,dr-d\omega'(d)\right) =\frac1d\int_0^d r[-\omega''(r)]\,dr\\
			&=\frac\beta d\int_0^d r\omega'(r)^2\,dr +\frac{2C_0}{d}\int_0^d\delta_0(2r)\omega'(r)\,dr.
	\end{aligned}
	\end{equation}
For the Dini term in the preceding estimate, since $\delta_0(2r)\ge\delta_0(d)$ for $\frac12d\le r\le d$
and  $\omega'$ is positive and decreasing, we have
	\[ \frac{2C_0}{d}\int_0^d\delta_0(2r)\omega'(r)\,dr \ge\frac{2C_0}{d} \int_{d/2}^d\delta_0(2r)\omega'(r)\,dr \ge C_0\delta_0(d)\omega'(d). \]
Hence, from \eqref{eq:ma_gain_identity}, and again, since $\omega'$ is decreasing, we obtain
	\begin{equation}\label{eq:ma_gain}
		\frac{\omega(d)}d-\omega'(d) \ge\frac{\beta d}{2}\omega'(d)^2 +C_0\delta_0(d)\omega'(d).
	\end{equation}
 
We next estimate the expression in square brackets in the second line of \eqref{eq:ma_closing_reserve}.
Since $\omega'(d)>T_0\ge2KM$ by Proposition \ref{prop32} and \eqref{eq:ma_constants}, we have
	\[ 2Kd\omega'(d)+4K^2Md \le4Kd\omega'(d) \le\frac{\beta d}{4}\omega'(d)^2, \]
where the last inequality follows from $\beta\ge16K$ by \eqref{eq:ma_constants} and $\omega'(d)>1$ by \eqref{eq:ma_contact_slope}.
Together with \eqref{eq:ma_gain}, this gives
	\begin{equation}\label{eq:ma_gain_after_geometry} 
		\frac{\omega(d)}d-\omega'(d) -2Kd\omega'(d)-4K^2Md \ge\frac{\beta d}{4}\omega'(d)^2 +C_0\delta_0(d)\omega'(d). 
	\end{equation}
Moreover, by \eqref{eq:ma_metric}, $u_{11}(x)\le C_0\omega'(d)$, and hence
	\begin{equation}\label{eq:ma_gain_lower1} 
		\frac1{u_{11}(x)} \left[\frac{\omega(d)}d-\omega'(d) -2Kd\omega'(d)-4K^2Md\right] \ge\frac{\beta d}{4C_0}\omega'(d)+\delta_0(d). 
	\end{equation}

We will also use the estimate
	\begin{equation}\label{eq:ma_gain_lower2}
		\frac1{u_{11}(x)}\left(\frac{\omega(d)}d-\omega'(d)\right) \ge\frac{\beta}{4C_0}\omega(d).
	\end{equation}
Indeed, if $\frac{\omega(d)}{d}\le2\omega'(d)$, this follows from \eqref{eq:ma_gain} and \eqref{eq:ma_metric}.
Otherwise, if $\frac{\omega(d)}{d}>2\omega'(d)$, then $\frac{\omega(d)}{d}-\omega'(d)>\frac{\omega(d)}{2d}$.
Note that \eqref{eq:ma_metric} and the last inequality in \eqref{eq:ma_riccati} imply $[d u_{11}(x)]^{-1} \ge \beta/C_0$.
We then obtain \eqref{eq:ma_gain_lower2}.

Substituting \eqref{eq:ma_gain_lower1} and \eqref{eq:ma_gain_lower2} in \eqref{eq:ma_closing_reserve},
and using $\omega'(d)>1$ by \eqref{eq:ma_contact_slope}, we obtain
	\begin{equation}\label{eq:ma_contradiction} 
		0\ge d\left[\frac{\beta\omega'(d)}{4C_0}-\frac{nK}{4}\right] +\omega(d)\left[\frac\beta{8C_0}-\frac{L_p}{2}\right] +K\rho(w)\left[\frac{3\beta}{C_0}-\frac{L_p}{2}\right]>0.
	\end{equation}
All three coefficients are strictly positive by \eqref{eq:ma_constants}.
This contradiction proves $Z\le0$.
\end{proof}

\appendix

\section{Determinant inequality}
\begin{lemma}\label{lem:ma_directional}
Let $X,W$ be positive definite symmetric $n\times n$ matrices, $F=X^{-1}$, and let $\nu\in\mathbb S^{n-1}$. Then
	\begin{equation}\label{eq:ma_directional_lemma}
 		\frac{\nu^TX\nu}{\nu^TW\nu}  \le\frac{\det X}{\det W}  \left(\frac{\tr (FW)}{n-1}\right)^{n-1}.
	\end{equation}
\end{lemma}

\begin{proof}
Take $Q=X^{-1/2}WX^{-1/2}$, and choose an orthonormal basis of $\mathbb R^n$ whose first vector is
$v_1=\frac{X^{1/2}\nu}{\sqrt{\nu^TX\nu}}$. 
In this basis the first diagonal entry of $Q$ equals $v_1^TQv_1=\frac{\nu^TW\nu}{\nu^TX\nu}$.
By Hadamard's determinant inequality,
	$$\det Q \le \prod_{i=1}^n Q_{ii} = \frac{\nu^TW\nu}{\nu^TX\nu} \prod_{i=2}^n Q_{ii}.$$
The arithmetic-geometric mean inequality then yields
	\begin{equation} 
 		\det Q\le\frac{\nu^TW\nu}{\nu^TX\nu} \left(\frac{\tr Q}{n-1}\right)^{n-1} =\frac{\nu^TW\nu}{\nu^TX\nu} \left(\frac{\tr(FW)}{n-1}\right)^{n-1}.
	\end{equation}
Since $\det Q=\det W/\det X$, the conclusion follows.
\end{proof}

\section{Matrix inequalities}

We first recall the standard mixed $\sigma_2$ inequality (see Wang \cite[\S2.3]{Wa09}). 
For symmetric matrices $H,K$ with $\lambda(H),\lambda(K)\in\Gamma_2$, the concavity of $\sigma_2^{1/2}$ on $\Gamma_2$ yields
	\begin{equation}\label{eq:mixed_sigma2_concavity}
 		\sqrt{\sigma_2(K)}\le\sqrt{\sigma_2(H)}  +\frac{S_2^{ij}(H)(K_{ij}-H_{ij})}{2\sqrt{\sigma_2(H)}},
	\end{equation}
where $S_2^{ij}(H)=(\tr H)\delta_{ij}-H_{ij}$. 
Since $S_2^{ij}(H)H_{ij}=2\sigma_2(H)$, we obtain
	\begin{equation}\label{eq:mixed_sigma2}
		 S_2^{ij}(H)K_{ij}\ge  2\sqrt{\sigma_2(H)\sigma_2(K)}.
	\end{equation}

The following lemma is the positive-source version of the directional comparison in Appendix~B of \cite{HLT}.

\begin{lemma}\label{lem:directional}
Let \(H,K\) be symmetric matrices with \(\lambda(H),\lambda(K)\in\Gamma_2\),
\(H_{11}>0\), and \(K_{11}\le L\), where \(L>0\). Then,
\begin{equation}\label{eq:matrixlemma}
 S_2^{ij}(H)K_{ij}
 \ge\frac{\sigma_2(K)H_{11}}{L+\sqrt{\sigma_2(K)}}.
\end{equation}
\end{lemma}

\begin{proof}
We first consider \(n\ge 3\).
Set
\[
 \theta=\frac{n-2}{n-1},\qquad
 E_0=\frac1{n-1}I-e_1\otimes e_1.
\]
By a straightforward computation, we have
\(\sigma_2(E_0)=-\frac\theta2\) and, for any symmetric matrix \(W\) and
\(a\in\mathbb R\),
\begin{align}
\label{eq-S_2(W+aE_0)}
\sigma_2(W+aE_0)=\sigma_2(W)+aW_{11}-\frac\theta2a^2.
\end{align}
Moreover,
\begin{align}
\label{eq-S_2^ij(W)E_0ij}
S_2^{ij}(W)(E_0)_{ij}=W_{11}.
\end{align}
We now claim that, for any \(W\in\Gamma_2\),
\(W+\frac{W_{11}}{\theta}E_0\in\Gamma_2\).
In fact, by taking \(a=W_{11}/\theta\) in
\eqref{eq-S_2(W+aE_0)}, we obtain
\begin{align}
\label{eq-S_2-W-shifted}
\sigma_2\left(W+\frac{W_{11}}{\theta}E_0\right)
=\sigma_2(W)+\frac{W_{11}^2}{2\theta}>0.
\end{align}
In addition,
\[
\operatorname{tr}\left(W+\frac{W_{11}}{\theta}E_0\right)
=\operatorname{tr}W+\frac{W_{11}}{n-2}>0,
\]
since \(\operatorname{tr}W>\big(\sum_{i,j}W_{ij}^2\big)^{1/2}
\ge |W_{11}|\) and \(n-2\ge1\).
This proves the claim.

We now apply \eqref{eq:mixed_sigma2} with \(H\) replaced by
\(H+\frac{H_{11}}{\theta}E_0\) and \(K\) replaced by
\(K+\frac{K_{11}}{\theta}E_0\), respectively.
Using \eqref{eq-S_2^ij(W)E_0ij} and \eqref{eq-S_2-W-shifted}, we get
\[
\begin{aligned}
S_2^{ij}(H)K_{ij}
&\ge
\frac{
\sqrt{(H_{11}^2+2\theta\sigma_2(H))
(K_{11}^2+2\theta\sigma_2(K))}
-H_{11}K_{11}}{\theta}\\
&\ge
\frac{H_{11}}{\theta}
\Big(\sqrt{K_{11}^2+2\theta\sigma_2(K)}-K_{11}\Big),
\end{aligned}
\]
where we used
\(\sqrt{H_{11}^2+2\theta\sigma_2(H)}>H_{11}>0\)
in the last inequality.
Rationalizing the expression, we then obtain
\begin{align}
\label{eq-S_2^ij(H)K_ij-final}
S_2^{ij}(H)K_{ij}
\ge
\frac{2H_{11}\sigma_2(K)}
{\sqrt{K_{11}^2+2\theta\sigma_2(K)}+K_{11}}.
\end{align}
Since \(\sigma_2(K)>0\), the function
\(f(t)=\sqrt{t^2+2\theta\sigma_2(K)}+t\) is strictly increasing,
by a direct differentiation.
By \(K_{11}\le L\) and \(0<\theta<1\), we get
\[
\begin{aligned}
\sqrt{K_{11}^2+2\theta\sigma_2(K)}+K_{11}
\le \sqrt{L^2+2\theta\sigma_2(K)}+L
\le 2\bigl(L+\sqrt{\sigma_2(K)}\bigr).
\end{aligned}
\]
Combining this with \eqref{eq-S_2^ij(H)K_ij-final}, we obtain
\eqref{eq:matrixlemma} for $n \ge 3$.

We next consider \(n=2\).
In this case, the matrices \(H,K\) are positive definite and
\(\sigma_2=\det\). Set
\[
H=\begin{pmatrix}
H_{11} & H_{12}\\
H_{12} & H_{22}
\end{pmatrix},
\qquad
K=\begin{pmatrix}
K_{11} & K_{12}\\
K_{12} & K_{22}
\end{pmatrix}.
\]
By a direct computation, we have
\[
\begin{aligned}
S_2^{ij}(H)K_{ij}
&=H_{22}K_{11}+H_{11}K_{22}-2H_{12}K_{12}\\
&=\frac{H_{11}}{K_{11}}
\left(K_{12}-\frac{H_{12}K_{11}}{H_{11}}\right)^2
+\frac{K_{11}}{H_{11}}\det H
+\frac{H_{11}}{K_{11}}\det K,
\end{aligned}
\]
and hence
\[
S_2^{ij}(H)K_{ij}
\ge \frac{H_{11}\det K}{K_{11}}
\ge \frac{H_{11}\det K}{L+\sqrt{\det K}},
\]
where we used \(0<K_{11}\le L\) and \(\det K>0\).
This proves the desired estimate.
\end{proof}


\begin{thebibliography}{99}

\bibitem{An15}
B. Andrews,
\emph{Moduli of continuity, isoperimetric profiles, and multi-point estimates in geometric heat equations},
in \emph{Surveys in differential geometry 2014. Regularity and evolution of nonlinear equations},
Surv. Differ. Geom., vol.~19, Int. Press, Somerville, MA, 2015, pp.~1--47.

\bibitem{AC11} 
B. Andrews and J. Clutterbuck,
\emph{Proof of the fundamental gap conjecture},
J. Amer. Math. Soc. \textbf{24} (2011), 899--916.

\bibitem{AL15}
B. Andrews and H. Li,
\emph{Embedded constant mean curvature tori in the three-sphere},
J. Differential Geom. \textbf{99} (2015), 169--189. 

\bibitem{Br13} 
S. Brendle,
\emph{Embedded minimal tori in $S^3$ and the Lawson conjecture},
Acta Math. \textbf{211} (2013), 177--190. 

\bibitem{Br14}
S. Brendle,
\emph{Two-point functions and their applications in geometry},
Bull. Amer. Math. Soc. \textbf{51} (2014), 581--596.

\bibitem{Ca90}
L. Caffarelli,
\emph{Interior $W^{2,p}$ estimates for solutions of Monge-Amp\`ere equations},
Ann. of Math. \textbf{131} (1990), 135--150. 

\bibitem{CHO16}
C. Chen, F. Han, and Q. Ou,
\emph{The interior $C^2$ estimate for the Monge-Amp\`ere equation in dimension $n=2$},
Anal. PDE \textbf{9} (2016), 1419--1432. 

\bibitem{CZZ}
R. Chen, X. Zhou, and R. Zhu,
\emph{Interior $C^{2,\alpha}$ regularity for the quadratic Hessian equation},
arXiv:2608.29484v2, 2026.

\bibitem{CW01}
K.-S. Chou and X.-J. Wang,
\emph{A variational theory of the Hessian equation},
Comm. Pure Appl. Math. \textbf{54} (2001), 1029--1064. 

\bibitem{CIL92}
M. G. Crandall, H. Ishii, and P.-L. Lions,
\emph{User's guide to viscosity solutions of second order partial differential equations},
Bull. Amer. Math. Soc. \textbf{27} (1992), 1--67. 

\bibitem{Fan26}
Z. Fan,
\emph{Hessian estimates for the $\sigma_2$ equation with variable right-hand side terms in dimension 4},
Adv. Math. \textbf{494} (2026), 110953.

\bibitem{GT01}
D.~Gilbarg and N.~Trudinger,
\emph{Elliptic Partial Differential Equations of Second Order},
Springer-Verlag, Berlin, 2001.  

\bibitem{GQ19}
P. Guan and G. Qiu,
\emph{Interior $C^2$ regularity of convex solutions to prescribing scalar curvature equations},
Duke Math. J. \textbf{168} (2019), 1641--1663.  

\bibitem{Ha92-1}
R. Hamilton,
\emph{An isoperimetric estimate for the Ricci flow on the two-sphere},
in \emph{Modern methods in complex analysis (Princeton, NJ, 1992)}, Ann. of Math. Stud., vol.~137,
Princeton Univ. Press, Princeton, NJ, 1995, pp.~191--200.

\bibitem{Ha92-2}
R. Hamilton,
\emph{Isoperimetric estimates for the curve shrinking flow in the plane},
in \emph{Modern methods in complex analysis (Princeton, NJ, 1992)}, Ann. of Math. Stud., vol.~137,
Princeton Univ. Press, Princeton, NJ, 1995, pp.~201--222.

\bibitem{He59}
E. Heinz,
\emph{On elliptic Monge-Amp\`ere equations and Weyl's embedding problem},
J. Analyse Math. \textbf{7} (1959), 1--52. 

\bibitem{Hu98}
G. Huisken,
\emph{A distance comparison principle for evolving curves},
Asian J. Math. \textbf{2} (1998), 127--133. 

\bibitem{HLT}
J.-Y. Hung, J. Liu, and N. Trudinger,
\emph{A new approach to interior estimates for fully nonlinear elliptic equations},
arXiv:2610.15457v2, 2026. 

\bibitem{JW07}
H.Y. Jian and X.-J. Wang,
\emph{Continuity estimates for the Monge-Amp\`ere equation},
SIAM J. Math. Anal. \text{39} (2007), 608--626.

\bibitem{JT1}
F. Jiang, N.S. Trudinger,
\emph{Oblique boundary value problems for augmented Hessian equations I},
Bull. Math. Sci. \textbf{8} (2018), 353--411.


\bibitem{LWarX}
Z. Li and K. Wu,
\emph{Interior Hessian estimates for the quadratic Hessian equation},
arXiv:2608.23233v2, 2026.

\bibitem{Liu21}
J. Liu,
\emph{Interior $C^2$ estimate for Monge-Amp\`ere equation in dimension two},
Proc. Amer. Math. Soc. \textbf{149} (2021), 2479--2486. 

\bibitem{LW15}
J. Liu and X.-J. Wang,
\emph{Interior a priori estimates for the Monge-Amp\`ere equation},
in \emph{Surveys in differential geometry 2014. Regularity and evolution of nonlinear equations},
Surv. Differ. Geom., vol.~19, Int. Press, Somerville, MA, 2015, pp.~151--177.

\bibitem{LT10}
J. Liu and N. Trudinger,
\emph{On Pogorelov estimates for Monge-Amp\`ere type equations},
Discrete Contin. Dyn. Syst. \textbf{28} (2010), 1121--1135. 

\bibitem{Mo25}
C. Mooney,
\emph{Remarks on the quadratic Hessian equation},
arXiv:2505.14586, 2025. 

\bibitem{Po78}
A. V. Pogorelov,
\emph{The Minkowski multidimensional problem},
V. H. Winston \& Sons, Washington, D.C.; Halsted Press, New York, 1978. 

\bibitem{Qiu24}
G. Qiu,
\emph{Interior Hessian estimates for $\sigma_2$ equations in dimension three},
Front. Math. \textbf{19} (2024), 577--598.

\bibitem{SY25}
R. Shankar and Y. Yuan,
\emph{Hessian estimates for the $\sigma_2$ equation in dimension four},
Ann. of Math. \textbf{201} (2025), 489--513.

\bibitem{Tr88}
N. S. Trudinger,
\emph{H\"older gradient estimates for fully nonlinear elliptic equations}
Proc. R. Soc. Edinburgh, Sect. A: Math. Phys. Sci. \textbf{108} (1988), 57-65.

\bibitem{Tr97}
N. S. Trudinger,
\emph{Weak solutions of Hessian equations},
Comm. Partial Differential Equations \textbf{22} (1997), 1251--1261. 

\bibitem{TW99}
N. S. Trudinger and X.-J. Wang,
\emph{Hessian measures. II},
Ann. of Math. \textbf{150} (1999), 579--604. 

\bibitem{Wa09}
X.-J. Wang,
\emph{The $k$-Hessian equation},
in \emph{Geometric analysis and PDEs}, Lecture Notes in Math., vol.~1977, Springer, 2009, pp.~177--252.

\bibitem{WY09}
M. Warren and Y. Yuan,
\emph{Hessian estimates for the $\sigma_2$ equation in dimension $3$},
Comm. Pure Appl. Math. \textbf{62} (2009), 305--321.

\bibitem{Yu23}
Y. Yuan,
\emph{A monotonicity approach to Pogorelov's Hessian estimates for Monge-Amp\`ere equation},
Math. Eng. \textbf{5} (2023), No. 037, 6pp. 

\end{thebibliography}
\end{document}